\documentclass[12pt,reqno]{amsart}

\usepackage[margin=1in]{geometry}
\usepackage{graphicx}
\usepackage{amsmath, amssymb, amsthm} 
\usepackage{subfiles}
\usepackage{newclude}

\theoremstyle{definition}
\newtheorem{theorem}{Theorem}
\newtheorem{example}{Example}
\newtheorem{corollary}{Corollary}
\newtheorem{proposition}{Proposition}
\newtheorem{remark}{Remark}
\newtheorem{definition}{Definition}
\newtheorem{thm-def}{Theorem-Definition}
\newtheorem{lemma}{Lemma}

\usepackage{lmodern}
\usepackage[T1]{fontenc}
\usepackage{fancyhdr}
\usepackage[svgnames]{xcolor} % If using tikz, put \usepackage{tikz} after this
\usepackage{pdfcolmk}
\usepackage{tikz-cd}
\usepackage{tikz}
\usepackage{pgfplotstable}
\pgfplotsset{compat=1.18}
\usepackage{mathtools}
\usepackage{mathrsfs}  

\usepackage{calligra}
\usepackage{ytableau}

\usepackage{caption}
\usepackage{subcaption}

\usepackage[maxbibnames=99,backend=biber, style=alphabetic]{biblatex}
\bibliography{splitting-loci-references}

\usepackage{hyperref}

\newcommand{\ceil}[1]{\left\lceil#1\right\rceil}
\DeclareMathOperator{\len}{len}
\newcommand{\ur}{\underline{r}}
\newcommand{\ud}{\underline{d}}
\newcommand{\cA}{\mathcal{A}}
\newcommand{\cV}{\mathcal{V}}
\newcommand{\cE}{\mathcal{E}}
\newcommand{\cF}{\mathcal{F}}
\newcommand{\cG}{\mathcal{G}}
\newcommand{\cO}{\mathcal{O}}
\newcommand{\cQ}{\mathcal{Q}}
\newcommand{\cS}{\mathcal{S}}
\newcommand{\cB}{\mathcal{B}}
\newcommand{\cR}{\mathcal{R}}
\newcommand{\cK}{\mathcal{K}}
\newcommand{\cW}{\mathcal{W}}
\newcommand{\bP}{\mathbb{P}}
\newcommand{\bU}{\mathbf{U}}
\newcommand{\bfS}{\mathbf{S}}
\newcommand{\Z}{\mathbb{Z}}
\newcommand{\T}{\mathbb{T}}

\newcommand{\Quot}{\operatorname{Quot}}
\newcommand{\Bal}{\operatorname{Bal}}
\newcommand{\FQuot}{\operatorname{FQuot}}
\newcommand{\univ}{\operatorname{univ}}
\newcommand{\Quott}{\Quot^{r,d}_{\bP^1}(\cO^{\oplus N})}
\newcommand{\Spec}{\operatorname{Spec}}
\newcommand{\Pic}{\operatorname{Pic}}
\newcommand{\HN}{\operatorname{HN}}
\newcommand{\Sec}{\operatorname{Sec}}
\newcommand{\Ext}{\operatorname{Ext}}
\newcommand{\End}{\operatorname{End}}
\newcommand{\Hom}{\operatorname{Hom}}
\newcommand{\id}{\operatorname{id}}
\newcommand{\rk}{\operatorname{rk}}
\newcommand{\Gr}{\operatorname{Gr}}

\newcommand{\Subsheaves}{\operatorname{Subsheaves}}

\newcommand{\codim}{\operatorname{codim}}

\newcommand{\pvec}[1]{\vec{#1}\mkern2mu\vphantom{#1}}
\newcommand{\splitt}{\overline{\Sigma}_{\vec{e}}}
\newcommand{\nyesplitt}{\nye{\Sigma}_{\vec{e}}}
\newcommand{\Splitt}{\mathbf{\overline{\Sigma}}_{\vec{e}}}

\newcommand{\SplittOpenF}{\mathbf{\Sigma}_{\vec{f}}}
\newcommand{\SplittNye}{\mathbf{\nye{\Sigma}}_{\vec{e}}}
\newcommand{\SplittNyeI}{\mathbf{\nye{\Sigma}}_{\vec{e}, I}}
\newcommand{\SplittOpen}{\mathbf{\Sigma}_{\vec{e}}}
\newcommand{\SplittNyeHN}{\mathbf{\nye{\Sigma}}_{\vec{e},\HN}}
\newcommand{\vbst}{\cB_{r,d}}

\newcommand{\nye}[1]{\widetilde{#1}}
\newcommand{\Aut}{\operatorname{Aut}}

\newcommand{\vectwo}[2]
{\begin{bmatrix} #1 \\ #2 \end{bmatrix}}

\newcommand{\cEnd}{\mathcal{E}nd}
\newcommand{\cHom}{\mathcal{H}om}

\usetikzlibrary{decorations.markings}
\makeatletter
\tikzcdset{
  open/.code     = {\tikzcdset{hook, circled};},
  closed/.code   = {\tikzcdset{hook, slashed};},
  open'/.code    = {\tikzcdset{hook', circled};},
  closed'/.code  = {\tikzcdset{hook', slashed};},
  circled/.code  = {\tikzcdset{markwith = {\draw (0,0) circle (.375ex);}};},
  slashed/.code  = {\tikzcdset{markwith = {\draw[-] (-.4ex,-.4ex) -- (.4ex,.4ex);}};},
  markwith/.code ={
    \pgfutil@ifundefined%
    {tikz@library@decorations.markings@loaded}%
    {\pgfutil@packageerror{tikz-cd}{You need to say %
      \string\usetikzlibrary{decorations.markings} to use arrows with markings}{}}{}%
    \pgfkeysalso{/tikz/postaction = {
      /tikz/decorate,
      /tikz/decoration={markings, mark = at position 0.5 with {#1}}}
    }
  },
}
\makeatother

\newsavebox{\pullback}
\sbox\pullback{%
\begin{tikzpicture}%
\draw (0,0) -- (1ex,0ex);%
\draw (1ex,0ex) -- (1ex,1ex);%
\end{tikzpicture}}
\title{Resolving the singularities of splitting loci} %Singularities of splitting loci %Resolving the singularities of splitting loci}
\author{Feiyang Lin}

\begin{document}

\begin{abstract}
We construct modular resolutions of singularities for splitting loci, and use them to show that tame splitting loci have rational singularities. As a corollary of our results and Hurwitz-Brill-Noether theory, we prove that if $C$ is a general $k$-gonal curve, the components of $W^r_d(C)$ have rational singularities. We also recover the classical Gieseker-Petri theorem. Along the way, we prove a cohomology vanishing statement for certain tautological vector bundles on $\Quott$, which may be of independent interest.
\end{abstract}

\maketitle

\tableofcontents

\section{Introduction} 
A vector bundle on a scheme $B \times \bP^1$ leads to a stratification of the base $B$ according to how the vector bundle splits when it is restricted to each fiber. The strata that arise are called splitting loci. The aim of this paper is to produce a modular resolution of singularities for these strata and use the resolution to study the singularities of splitting loci in the case of the universal family. 

Let $\vbst$ be the moduli stack of rank $r$, degree $d$ vector bundles on $\bP^1$. The data of how a vector bundle of rank $r$ and degree $d$ splits is captured by a splitting type $\vec{e} = (e_1, \dots, e_r)$, where $e_1 \leq e_2 \leq \cdots \leq e_r$ and $\sum_{i=1}^r e_i = d$. We define $\SplittOpen$ to be the locally closed substack in $\vbst$ parametrizing families of vector bundles each with splitting type $\vec{e}$, and define $\Splitt$ as the closure of $\SplittOpen$ in $\vbst$. For convenience, we write $\cO(\vec{e}) \coloneq \cO(e_1) \oplus \cdots \oplus \cO(e_r)$.

\begin{definition}
A splitting type $\vec{e} = (e_1, \dots, e_r)$ is balanced if $|e_r - e_1| \leq 1$, or equivalently if $h^1 \cEnd(\cO(\vec{e})) = 0$. If there exists some $1 \leq k \leq r$ such that $\vec{f} = (e_1, \dots, e_k)$ and $\vec{g} = (e_{k+1}, \dots, e_r)$ are both balanced, then  we say that the splitting type $\vec{e}$ is tame.
\end{definition}

For example, all rank two splitting types are tame. $(1,2,5,6)$ is tame, but $(0,2,4)$ is not. Recall that a scheme $X$ has rational singularities if there exists a resolution of singularities $\rho: \nye{X} \to X$ such that $R^{> 0} \rho_* \cO_{\nye{X}} = 0$ and the natural map $\cO_X \to \rho_* \cO_{\nye{X}}$ is an isomorphism. One checks that having rational singularities is a smooth-local property, in the sense of \cite[\href{https://stacks.math.columbia.edu/tag/0348}{Tag 0348}]{stacks-project}. So it makes sense to say that an algebraic stack has rational singularities.

Our main theorem is the following.
\begin{theorem} \label{thm:rational} Let $k$ be an algebraically closed field of characteristic $0$. Over $\Spec k$, for tame splitting types $\vec{e}$, $\Splitt$ has rational singularities.
\end{theorem}

Combined with Hurwitz-Brill-Noether theory \cite{Larson2021}, \cite{LLV}, our result implies the following corollary.
\begin{corollary} Let $C$ be a general $k$-gonal curve. Then components of $W^r_d(C)$ have rational singularities.
\end{corollary}
\begin{proof}
\cite{Larson2021} proves that for a general degree $k$ cover $f: C \to \bP^1$, the morphism $\Pic_d(C) \to \cB_{k, d-g+1-k}$ induced by pushing forward line bundles along $f$ is smooth. The same paper plus the work of \cite{Coppens2025} prove that the components of $W^r_d(C)$ correspond to the maximal splitting types with at least $r+1$ sections, which are tame.
\end{proof}

The construction of a modular resolution of singularities for $\Splitt$ is essential to the proof of Theorem \ref{thm:rational}. To describe the resolution, for $\ur = (r_1, \dots, r_m), \ud = (d_1, \dots, d_m)$, let us define the hyper-Quot stack
$\FQuot(\vbst, \cE_{\univ}; \ur, \ud) \to \vbst$ to be the algebraic stack representing the following functor: 
\begin{align*}
  \left (T \to \Spec k \right ) 
  & \mapsto 
  \left \{
    \begin{tabular}{lll}
      \text{A vector bundle $\cE$ of rank $r$ and fiberwise degree $d$  } \\
      \text{and a flag of subsheaves $\cE_1 \hookrightarrow \cE_2 \hookrightarrow \dots \hookrightarrow \cE_{m+1} = \cE$ on $T \times \bP^1$ } \\
      \text{such that $\cE_{i}$ is locally free of rank $r-r_i$ and fiberwise} \\
      \text{degree $d-d_i$, and each quotient $\cE/\cE_i$ is flat over $T$}
    \end{tabular}
  \right\} / \sim. 
\end{align*}

\begin{theorem} 
\label{thm:smooth1}
Let $E = \cO(\vec{e})$ be a vector bundle on $\bP^1$, and let 
\[
  E_1 \subset E_2 \subset \cdots \subset E_{m+1} = E
\]
be a sequence of subbundles such that for all $1 \leq i \leq m$, each quotient sheaf $E_{i+1}/E_i$ is a nonzero balanced vector bundle, and up to automorphisms of $E_i$, the inclusion $E_i \subset E_{i+1}$ is the unique injective map $E_i \hookrightarrow E_{i+1}$. Let $\ur = (r_1, \dots, r_m), \ud = (d_1, \dots, d_m)$ be the ranks and degrees of $E/E_i$. Then $\FQuot(\vbst,\cE_{\univ}; \ur, \ud)$ is a smooth and irreducible algebraic stack proper over $\vbst$. It maps surjectively onto $\Splitt$ and it is an isomorphism over the dense open $\SplittOpen$.
\end{theorem}
Note that the Harder-Narasimhan flag of subbundles of $\cO(\vec{e})$ always satisfies the conditions of Theorem \ref{thm:smooth1}, but for any given $\vec{e}$, there may be more than one choice (see Example \ref{ex:e-admissible-sets}). 

Combined with Hurwitz-Brill-Noether theory, Theorem \ref{thm:smooth1} recovers the classical Gieseker-Petri theorem.

\begin{corollary}[Gieseker-Petri] Let $C$ be a general curve of genus $g$. Then $G^r_d(C)$ is smooth. 
\end{corollary}
\begin{proof}
Let $k$ be sufficiently large and let $f: C \to \bP^1$ be a general cover of degree $k$ and genus $g$. Then by Hurwitz-Brill-Noether theory, $W^r_d(C) = W^{\vec{e}_{\max, r}}(C)$, where \[\vec{e}_{\max, r} = ((-2)^{g-d+r}, (-1)^{k-g+d-2r-1}, 0^{r+1}),\] and exponents mean repeated entry. Moreover, for any line bundle $L \in W^r_d(C)$, we know that $f_* L = \cO(\vec{e})$, where 
\[\vec{e} = ((-2)^{h^1(L)}, (-1)^{k-h^1(L)-h^0(L)}, 0^{h^0(L)}),\] and $h^0(L) \geq r+1$. So to give an $(r+1)$-dimensional subspace $V \subseteq H^0(C, L)$ is the same as giving an inclusion $\cO^{\oplus (r+1)} \hookrightarrow f_*L$. Therefore, the relative Quot construction corresponding to the one-step inclusion
$
  \cO^{\oplus (r+1)} \subseteq \cO(\vec{e}_{\max, r}),
$
when base-changed to $\Pic_d(C)$, is precisely $G^r_d(C)$. Hence Theorem \ref{thm:smooth1} implies that $G^r_d(C)$ is smooth.
\end{proof}

Tame splitting loci admit a resolution of singularities by a relative Quot construction, and do not require a hyper-Quot construction, which is necessary in general. To show that the higher derived pushforward of the structure sheaf vanishes (and thus $\Splitt$ has rational singularities), we embed this resolution inside a fiber bundle of Quot schemes $\Quott$, where it is cut out by the section of a vector bundle. This reduces the vanishing of higher pushforwards to a cohomology vanishing statement for certain tautological bundles on $\Quott$. While this approach parallels the classical proof that determinantal varieties have rational singularities, the lack of a Borel-Weil-Bott type statement for $\Quot$ schemes presents substantial challenges at this step. To access the desired cohomology vanishing statement, we appeal to Str{\o}mme's embedding of Quot schemes on $\bP^1$ into a product of Grassmannians and further reduce the problem to a cohomology vanishing statement for certain tautological vector bundles on a product of Grassmannians. The desired statement then follows from an application of Borel-Weil-Bott and an involved combinatorial analysis.

Quot schemes on $\bP^1$ were first systematically studied by Str{\o}mme \cite{Stromme1987}, where they showed that $\Quott$ is a nonsingular, irreducible, projective, rational variety. More recently, motivated by analogies to the Hilbert scheme of points on a surface and computations of the quantum cohomology of the Grassmannian, there has been a lot of research towards understanding the cohomology of tautological bundles on Quot schemes on curves, such as \cite{Oprea_Sinha_2022}, \cite{Marian_Oprea_Sam_2023}, \cite{Marian_Negut_2024}, \cite{Sinha_Zhang_2024}. However, most results in the literature focus on punctual Quot schemes (the case where $r = 0$), whereas our result applies to Quot schemes parametrizing higher rank quotients. 

Here is the cohomology vanishing theorem we prove. Let $d > 0$ and $N \geq r \geq 0$. Let $\cQ$ be the universal quotient on $\Quott \times \bP^1$, and let $p: \Quott \times \bP^1 \to \Quott$ be the projection to the first factor. Then for $m \geq -1$, $p_* \cQ(m)$ is a vector bundle on $\Quott$. Let $\cE_m = (p_* \cQ(m))^{\vee}$. Given a partition $\lambda$, we write $\bfS_{\lambda}$ to denote the Schur functor associated to $\lambda$. When $\lambda = (1^k)$, $\bfS_{\lambda} \cE = \wedge^k \cE$; and when $\lambda = (k)$, $\bfS_{\lambda} \cE = \text{Sym}^k \cE$.

\begin{theorem} \label{thm:general_vanishing}
Let $m \geq -1$ be an integer. Let $\lambda_{-1}, \lambda_0, \lambda_1, \dots, \lambda_m$ be partitions, at least one of which is nonempty. Let $D = |\lambda_{-1}| + \cdots + |\lambda_m|$. Then 
\[
  H^{> D}(\Quott, 
  \otimes_{i = -1}^m \bfS_{\lambda_i} \cE_i) = 0.
\]
\end{theorem}

\noindent
{\textbf{Acknowledgements.}} 
I thank David Eisenbud and Hannah Larson for many fruitful discussions and insights about this work. Thanks also to  Alina Marian, Martin Olsson, Claudiu Raicu and Shubham Sinha for helpful conversations. Part of this work was conducted while the author was supported by the Jane Street Graduate Research Fellowship, and NSF Grant DMS 2001649. 

%!TEX root = splitting_loci.tex
\section{Splitting loci and vector bundles on \texorpdfstring{$\bP^1$}{P1}} \label{sec:preliminaries_on_splitting_loci}
In this section, we review some facts about splitting loci that will be useful later.

Recall that $\vbst$ is defined by the following functor:
\[
  \left (T \to \Spec k \right ) \mapsto 
  \left \{
    \begin{tabular}{lll}
      \text{Vector bundle $\cE$ on $T \times \bP^1$, } \\
      \text{such that $\cE_t$ has rank $r$ and} \\
      \text{degree $d$ for all $t \in T$} \\
    \end{tabular}
  \right\}. 
\] 
When we have a scheme $B$ equipped with a vector bundle $\cE$ on $B \times \bP^1$ of rank $r$ and fiberwise degree $d$, this data gives rise to a morphism $\phi_{\cE}: B \to \vbst$. We write $\Sigma_{\vec{e}}(\cE) \subset \splitt(\cE) \subset B$ for the pullback of $\Splitt$ and $\SplittOpen$ along $\phi_{\cE}$. In general, $\splitt(\cE)$ will no longer be the closure of $\Sigma_{\vec{e}}(\cE)$ in $B$, but it still has a nice description in terms of the dominance order between splitting types. Given two splitting types $\vec{e}, \vec{f}$, we say that $\vec{e} \geq \vec{f}$ if $e_1 + \cdots + e_j \geq f_1 + \cdots + f_j$ for each $1 \leq j \leq r$. Equivalently, $\vec{e} \geq \vec{f}$ if $e_j + \cdots + e_r \leq f_j + \cdots + f_r$ for each $1 \leq j \leq r$.

This partial order, also called the dominance order, turns out to capture when $\Splitt \supset \mathbf{\Sigma}_{\vec{f}}$ in $\vbst$, as will be explained in Proposition \ref{prop:equiv_leq}. In terms of this partial order, we have $\splitt(\cE) = \bigcup_{\vec{f} \leq \vec{e}} \Sigma_{\vec{f}}(\cE)$. We write $u(\vec{e}) = h^1 \cEnd(\cO(\vec{e}))$, which is the expected codimension of $\splitt$. Note also that given a fixed rank and degree, there is a unique balanced splitting type, which is maximal with respect to dominance. When the degree is divisible by rank, the unique balanced vector bundle of that rank and degree is called perfectly balanced. This corresponds to a splitting type $\vec{e} = (e_1, \dots, e_r)$ where $e_1 = e_r$. Given a vector bundle $E$ on $\bP^1$, we will write $\mu(E) \coloneq \frac{\deg E}{\rk E}$ for its slope.

\subsection{Closure order on splitting types}
We collect several equivalent conditions describing when one splitting locus lies in the closure of another.
\begin{definition} Let $\vec{e}$, $\vec{f}$ be two splitting types of the same rank and degree. Let the Harder-Narasimhan flag of subbundles of $\cO(\vec{e})$ be $E_1 \subset E_2 \subset \cdots \subset E_m \subset E_{m+1} = \cO(\vec{e})$, and choose a splitting $\cO(\vec{f}) = Q_1 \oplus Q_2 \oplus \cdots \oplus Q_{n+1}$, where each $Q_i$ is a perfectly balanced vector bundle and $\mu(Q_i) < \mu(Q_{i+1})$. We say that $\cO(\vec{f})$ admits a flag of subbundles of type $\vec{e}$ if there exists a flag of subbundles $E_1' \subset E_2' \subset \cdots \subset E_m' \subset E_{m+1}' = \cO(\vec{f})$ such that $\rk E_i' = \rk E_i$ and $\deg E_i' = \deg E_i$. 
An extension of type $\vec{f}$ is a sequence of extensions
$0 \to K_i \to K_{i+1} \to Q_{i+1} \to 0$,
where $1 \leq i \leq n$, and $K_1 = Q_{1}$. We say that $\cO(\vec{e})$ can be realized as an extension of type $\vec{f}$ if there exists an extension of type $\vec{f}$ where $K_{n+1} \cong \cO(\vec{e})$.
\end{definition}
\begin{example}
Let $\vec{f} = (-2,2)$. Then $Q_1 = \cO(2)$, $Q_2 = \cO(-2)$. For $\cO(\vec{e})$ to be realizable as an extension of type $\vec{f}$ means there exists a short exact sequence
  $0 \to \cO(-2) \to \cO(\vec{e}) \to \cO(2) \to 0.$
This occurs if and only if $\vec{e} = (0,0), (-1,1), (-2,2)$.
\end{example}

\begin{lemma} 
\label{lem:subsheaf} 
Let $\vec{e}$, $\vec{f}$ be two splitting types of the same rank and degree. Let $\vec{e} = (\text{Bal}, \vec{a})$, where $\Bal$ is a balanced splitting type with largest summand $m_o$ and $a_1 > m_o$. Then $\vec{f} \leq \vec{e}$ if and only if there exists $\pvec{a}' \leq \vec{a}$ and an inclusion $\cO(\pvec{a}') \hookrightarrow \cO(\vec{f})$.
\end{lemma}
\begin{proof}
Let $i$ be the length of $\Bal$, and let $s = r-i$ be the length of $\vec{a}$.

($\Rightarrow$) Let $\pvec{a}' = (\deg \vec{a} - \sum_{j = 2}^{s}f_{i+j}, f_{i+2}, f_{i+3},\dots, f_r)$. Then since $\vec{f} \leq \vec{e}$, we see that $\pvec{a}' \leq \vec{a}$ and $a_1' = \deg \vec{a} - \sum_{j = 2}^{s}f_{i+j} \leq f_{i+1}$. Since $a_j' \leq f_{i+j}$ for all $j$, there is an inclusion $\cO(\pvec{a}') \hookrightarrow \cO(\vec{f})$.

($\Leftarrow$) Suppose that we have $\cO(\pvec{a}') \hookrightarrow \cO(\vec{f})$. We claim that $a_{s-j}' \leq f_{r-j}$ for all $0 \leq j < s$. Suppose towards a contradiction that $a_{s-j}' > f_{r-j}$. Then there are no nonzero homomorphisms $\cO(a_{s-j}',\dots, a_s')$ to $\cO(f_1,\dots,f_{r-j})$. But then we must have $\cO(a_{s-j}',\dots, a_s') \hookrightarrow \cO(f_{r-j+1},\dots, f_r)$. A rank $j+1$ vector bundle cannot be a subsheaf of a rank $j$ vector bundle, so we get a contradiction.

Having proved the claim, it follows that for all $0 \leq j \leq s-1$,
\[
  f_{r-j} + \cdots + f_r \geq a_{s-j}' + \cdots + a_s' \geq a_{s-j} + \cdots + a_s = e_{r-j} + \dots + e_r,
\]
or equivalently $\sum_{j = 1}^{k} f_j \leq \sum_{j = 1}^k e_j$ for all $k \geq i$.
It remains to show that $f_1 + \cdots + f_k \leq e_1 + \cdots + e_k$ for $k < i$. Let $\vec{g} = (f_1, \dots, f_{i-1}, g_i)$, where $g_i = \deg \Bal - (f_1 + \cdots + f_{i-1})$. Then $\deg \vec{g} = \deg \Bal$, and $g_i \geq f_i$. Since $\Bal$ is balanced, $\vec{g} \leq \Bal$. So for $k < i$,
\[
  f_1 + \cdots + f_k = g_1 + \cdots + g_k \leq \Bal_1 + \cdots + \Bal_k = e_1 + \cdots + e_k.
\]
\end{proof}

\begin{proposition} 
\label{prop:equiv_leq}
Let $E = \cO(\vec{e}), F = \cO(\vec{f})$ be vector bundles on $\bP^1$ of the same degree and rank. Then the following are equivalent:
\begin{enumerate}
  \item there exists a family of vector bundles on $\bP^1_{k[t]}$ whose restriction to the fiber at $0$ is $\cO(\vec{f})$ and away from $0$ is $\cO(\vec{e})$;
  \item $\Splitt \supset \mathbf{\Sigma}_{\vec{f}}$;
  \item $h^1 \cO(\vec{e})(m) \leq h^1 \cO(\vec{f})(m)$ for all $m \in \Z$;
  \item $\sum_{i = 1}^{k} e_i \geq \sum_{i = 1}^k f_i$ for all $1 \leq k \leq r$, i.e. $\vec{e} \geq \vec{f}$;
  \item $\cO(\vec{f})$ admits a flag of subbundles of type $\vec{e}$;
  \item $\cO(\vec{e})$ is realizable as an extension of type $\vec{f}$.
\end{enumerate}
\end{proposition}
\begin{proof}
We will prove $(1) \Rightarrow (2) \Rightarrow (3) \Rightarrow (4) \Leftrightarrow (5) \Rightarrow (2)$, and delay the proof of $(2) \Rightarrow (6) \Rightarrow (1)$ to the next subsection.

(1) $\Rightarrow$ (2): clear.

(2) $\Rightarrow$ (3): this follows from upper-semicontinuity.

(3) $\Rightarrow$ (4): Suppose for contradiction that there exists some $k$ such that $\sum_{i = 1}^k f_i > \sum_{i = 1}^k e_i$. Let $k$ be the maximal index for which this occurs. Then $\sum_{i = k+1}^r f_i < \sum_{i = k+1}^r e_i$ but $\sum_{i = k+2}^r f_i \geq \sum_{i = k+2}^r e_i$. This implies that $f_{k+1} < e_{k+1}$. Now $h^0 \cO(\vec{f})(-f_{k+1}-1) = f_{k+1} + \cdots + f_r - (r-k)f_{k+1} < e_{k+1} + \cdots + e_r - (r-k)f_{k+1} \leq h^0 \cO(\vec{e})(-f_{k+1}-1)$. Equivalently, $h^1 \cO(\vec{f})(-f_{k+1}-1) \leq h^1 \cO(\vec{e})(-f_{k+1}-1)$, which contradicts (3). 

(4) $\Leftrightarrow$ (5): We proceed by induction on $k$. When $k = 1$, $\vec{e} = \vec{f}$ and both (4) and (5) are trivial. Now suppose that (4) $\Leftrightarrow$ (5) is true for all ranks less than $k$. Write $\vec{e} = (m_o, \dots, m_o, \vec{a})$ where $a_1 > m_o$. Then Lemma \ref{lem:subsheaf} shows that $\vec{e} \geq \vec{f}$ if and only if there exists a subbundle $\cO(\pvec{a}') \hookrightarrow \cO(\vec{f})$ with $\pvec{a}' \leq \vec{a}$. By induction, $\pvec{a}'$ admits a flag of subbundles of type $\vec{a}$ if and only if $\pvec{a}' \leq \vec{a}$. Adjoining the inclusion $\cO(\pvec{a}') \hookrightarrow \cO(\vec{f})$ to this flag proves the desired equivalence.

(5) $\Rightarrow$ (2): The proof of this part depends on results in future sections. We show that there exists an irreducible algebraic stack parametrizing families of vector bundles equipped with flags of subbundles of type $\vec{e}$, which maps properly to $\vbst$. The uniqueness of the Harder-Narasimhan flag for $\cO(\vec{e})$ shows that this is an isomorphism over $\SplittOpen$, and the irreducibility of this stack shows that the image is exactly the closure $\Splitt$. 
\end{proof}

\subsection{An explicit smooth cover of \texorpdfstring{$\vbst$}{the stack of vector bundles on P1}} 
\label{subsec:sm_cover}
In the literature, the algebraicity of the stack of vector bundles on a fixed projective scheme is usually proven by taking open loci in increasingly large Quot schemes. In the case of $\bP^1$, there is a nice family of affine spaces which forms a smooth cover of $\vbst$. This smooth cover will be used in Section \ref{sec:construction_of_the_resolution} to prove irreducibility and smoothness of the resolution of singularities we construct.

Let $U_{\vec{f}}$ be the affine space corresponding to the vector space $H^1 \cEnd (\cO(\vec{f}))$. After fixing an isomorphism $\cO(\vec{f}) \cong Q_1 \oplus Q_2 \oplus \cdots \oplus Q_{m+1}$, where each $Q_i$ is perfectly balanced and $\mu(Q_1) \leq \mu(Q_2) \leq \cdots \leq \mu(Q_{m+1})$, we may identify $H^1 \cEnd (\cO(\vec{f})) \cong \oplus_{j = 2}^{m} \Ext^1_{\bP^1}(Q_j, \oplus_{i = 1}^{j-1} Q_i)$. 

\begin{lemma} Let $A, B$ be vector bundles on $\bP^1$, and let $0 \to A \to E \to B \to 0$ be an extension. Then $h^1 E(m) \leq h^1 (A \oplus B)(m)$ for all $m \in \Z$.
\end{lemma}
\begin{proof}
The long exact sequence in cohomology induced by 
\[
  0 \to A(m) \to E(m) \to B(m) \to 0
\]
shows that 
  $h^1 E(m) \leq h^1 A(m) + h^1 B(m) = h^1(A\oplus B)(m)$.
\end{proof}
Combined with $(2) \Leftrightarrow (3) \Leftrightarrow (4)$ from Proposition \ref{prop:equiv_leq}, this implies that 
\begin{lemma} \label{lem:ext_leq}
Let $\vec{a}, \vec{b}$ be two splitting types of ranks $s$ and $r-s$, let $\cO(\vec{f}) = \cO(\vec{a}) \oplus \cO(\vec{b})$, and let $0 \to \cO(\vec{a}) \to \cO(\vec{e}) \to \cO(\vec{b}) \to 0$ be an extension. Then $\vec{e} \geq \vec{f}$. \qed
\end{lemma}

\begin{lemma} \label{lem:ext_iterated_leq} Let $\vec{e}$ be realizable as an extension of type $\vec{f}$. Then $\vec{e} \geq \vec{f}$. In particular, if $\cO(\vec{f}) = Q_1 \oplus \cdots \oplus Q_{m+1}$ is a decomposition into perfectly balanced bundles such that $\mu(Q_i) < \mu(Q_{i+1})$, then $e_r \leq \mu(Q_{m+1})$ and $\Ext^1(\cO(\vec{e}), Q_{m+1}) = 0$. If $Q$ is perfectly balanced with $\mu(Q) > \mu(Q_{m+1})$, then $\dim \Ext^1(Q, \cO(\vec{e}))$ does not depend on $\vec{e}$.
\end{lemma}
\begin{proof}
The implications are clear from the statement that $\vec{e} \geq \vec{f}$. To prove this, we proceed by induction on $m$. If $m = 1$, then this is Lemma \ref{lem:ext_leq}. Let $\vec{a} = Q_1 \oplus \cdots \oplus Q_m$. If $\vec{e}$ is realizable as an extension of type $\vec{f}$, then there exists an exact sequence 
\[
  0 \to \cO(\vec{b}) \to \cO(\vec{e}) \to Q_{m+1} \to 0,
\]
where $\cO(\vec{b})$ is realizable as an extension of type $\vec{a}$. Lemma \ref{lem:ext_leq} implies that $\vec{e} \geq \pvec{f}'$, where $\cO(\pvec{f}') = \cO(\vec{b}) \oplus Q_{m+1}$. By induction, $\vec{b} \geq \vec{a}$ and the largest summand of both are less than $\mu(Q_{m+1})$. So $\pvec{f}' \geq \vec{f}$. Putting the inequalities together, we get $\vec{e} \geq \vec{f}$.
\end{proof} 

\begin{proposition} \label{prop:moduli_construction}Let $\mathbf{U}_{\vec{f}}$ be the functor defined as follows. Let $T$ be a $k$-scheme, and let $p: T \times \bP^1 \to \bP^1$. A morphism $T \to \mathbf{U}_{\vec{f}}$ is the data of vector bundles $\cK_1, \cK_2, \dots, \cK_{m+1}$ on $T \times \bP^1$, the data of extensions
\[
  0 \to \cK_i \to \cK_{i+1} \to p^* Q_{i+1} \to 0,
\]
plus an isomorphism $\cK_1 \cong p^* Q_1$. Then there exist a sequences of stacks with maps $\bU_{\vec{f}}^{(0)} \to \bU_{\vec{f}}^{(1)} \to \cdots \to \bU_{\vec{f}}^{(m+1)}$, so that $\bU_{\vec{f}}^{(0)} \to \bU_{\vec{f}}^{(1)}$ is a $\text{GL}_{\rk Q_1}$-torsor, $\bU_{\vec{f}}^{(i)} \to \bU_{\vec{f}}^{(i+1)}$ where $i \geq 1$ is an open locus inside a vector bundle over $\bU_{\vec{f}}^{(i+1)}$, $\bU_{\vec{f}}^{(0)} \cong \bU_{\vec{f}}$, and 
$\bU_{\vec{f}}^{(m+1)} = \bigcup_{\vec{e} \geq \vec{f}} \SplittOpen \subset \vbst$. 
In particular, the map $\mathbf{U}_{\vec{f}} \to \vbst$ is surjective onto $\bigcup_{\vec{e} \geq \vec{f}} \SplittOpen$. Moreover, $\mathbf{U}_{\vec{f}}$ is an algebraic space. 
\end{proposition}
\begin{proof}
First we explain the construction of $\bU_{\vec{f}}^{(i)}$. 
We start with $\bU_{\vec{f}}^{(m+1)} = \bigcup_{\vec{e} \geq \vec{f}} \SplittOpen$. Let $T$ be a $k$-scheme with maps $p: T \times \bP^1 \to \bP^1$ and $\pi: T \times \bP^1 \to T$. By induction, we assume that a morphism $T \to \bU_{\vec{f}}^{(k)}$ is the data of extensions
\[
  0 \to \cK_i \to \cK_{i+1} \to p^* Q_{i+1} \to 0
\]
where $k \leq i \leq m$, such that $R^1 \pi_* \cHom(\cK_{k}, p^* Q_{k}) = 0$ and $\cK_{m+1}$ splits as at worst $\vec{f}$. Then we define 
\[
\bU_{\vec{f}}^{(k-1)} = \pi_* \cHom(\cK_{k}, p^* Q_{k}) \setminus (\mathcal{Z}_1 \cup \mathcal{Z}_2),
\] where $\mathcal{Z}_i$ are defined as follows. Let $\cK_{k-1}$ be the kernel of the universal map $\cK_{k} \to p^*Q_{k}$ on $\pi_* \cHom(\cK_{k}, p^* Q_{k}) \times \bP^1$. Then $\mathcal{Z}_1 = \text{Supp} R^1 \pi_* \cHom(\cK_{k-1}, p^* Q_{k-1})$, and $\mathcal{Z}_2$ is the closed locus where the map $\cK_{k-1} \to p^* Q_{k-1}$ fails to be surjective. 

Continue this process until we have constructed $\bU_{\vec{f}}^{(1)}$. Then $\bU_{\vec{f}}^{(0)}$ is defined to be $\underline{\text{Isom}}(\cK_1, p^*Q_1)$ over the open locus in $\bU_{\vec{f}}^{(1)}$ where $\cK_1$ is isomorphic to $Q_1$. This is open because $Q_1$ is perfectly balanced. The functor of points of $\bU_{\vec{f}}^{(0)}$ is that of $\bU_{\vec{f}}$, with the extra conditions that $R^1 \pi_* \cHom(\cK_i, p^* Q_i) = 0$ for all $2 \leq i \leq m+1$ and $\cK_{m+1}$ splits as at worst $\vec{f}$. However, these extra conditions are also automatically satisfied by points of $\bU_{\vec{f}}$ due to Lemma \ref{lem:ext_iterated_leq}. So $\bU_{\vec{f}}^{(0)} \cong \bU_{\vec{f}}$. Since $\bU_{\vec{f}} \to \vbst$ is flat and locally finitely presented, it is open. The modular nature of the image of $\bU_{\vec{f}}$ shows that the image must be a union of $\SplittOpen$, where $\vec{e} \geq \vec{f}$ due to Lemma \ref{lem:ext_iterated_leq}. We also know that its image contains $\mathbf{\Sigma}_{\vec{f}}$. So $\bU_{\vec{f}} \to \vbst$ must be surjective onto $\bigcup_{\vec{e} \geq \vec{f}} \SplittOpen$.

To show that $\bU_{\vec{f}}$ is an algebraic space, by Theorem 2.2.5 of \cite{Conrad_2007}, it suffices to check that the geometric points have no nontrivial automorphisms. Given an extension $k \to k'$, a $k'$-point of $\bU_{\vec{f}}$ is the data of vector bundles $K_i$ on $\bP^1_{k'}$, the data of extensions
\[
  0 \to K_i \to K_{i+1} \to Q_{i+1} \otimes_k k' \to 0
\]
for each $1 \leq i \leq m$, and an isomorphism $K_1 \cong Q_1 \otimes_k k'$. The automorphism group of this point is the product of the autormophism groups of each extension. So it suffices to show that each extension has no automorphisms. Lemma \ref{lem:ext_automorphism} below shows that since $\Hom_{\bP^1_{k'}}(Q_{i+1} \otimes_k k', K_{i}) = 0$ for all $1 \leq i\leq m$, the extensions have no nontrivial automorphisms.
\end{proof}

\begin{lemma} \label{lem:ext_automorphism}
Let $X$ be a proper variety over $k$, with coherent sheaves $A, B$, and an extension
\[
  \xi: 0 \to A \xrightarrow{i} E \xrightarrow{q} B \to 0
\]
on $X$. Then there is a map $\Hom(B, A) \to \Hom(E, E)$, by precomposing with $q$ and postcompositing with $i$. There is also a map $\Aut(\xi) \to \Hom(E,E)$ defined by identifying an automorphism of the extension with $\phi: E \to E$, and sending it to $\phi - \operatorname{id}_E$. Both maps are injective and maps onto the same image. In particular, if $\Hom(B, A) = 0$, then $\Aut(\xi)$ is trivial.
\end{lemma}
\begin{proof}
\[
  \begin{tikzcd}
0 \arrow[r] & A \arrow[r, "i"] \arrow[d, no head, double line] & E \arrow[r, "q"] \arrow[d, "\phi"] & B \arrow[r] \arrow[d, no head, double line] & 0 \\
0 \arrow[r] & A \arrow[r, "i"]                                & E \arrow[r, "q"]                   & B \arrow[r]                                & 0
\end{tikzcd}
\]
Since $q \phi = q$, we have $q(\phi - \id_E) = 0$. Since $0 \to \Hom(E,A) \to \Hom(E, E) \to \Hom(E, B)$ is exact, we know that $q \phi = q$ if and only if $\phi-\id_E$ factors through $i: A \to E$. Similarly, $\phi i = i$ if and only if $\phi-\id_E$ also factors through $q: E \to B$. 
\end{proof}

Let $X$ be proper over $\Spec k$, and let $p: X \times \Ext^1_X(A, B) \to X$. Then as explained on Page 118 of \cite{lePotier}, there is a universal extension $0 \to p^* A \to \cE \to p^* B \to 0$ on $X \times_k \Ext^1_X(A, B)$. The $T$-points of $\Ext^1_X(A,B)$ is the set of extensions $0 \to p^* A \to E \to p^* B \to 0$ on $T \times_k X$ modulo isomorphisms. In particular it is the coarse space of the moduli space of extensions of $B$ by $A$, i.e. it has the property that if one has a family of extensions of $A$ by $B$ on $T \times_k X$, then there is a unique map $T \to \Ext^1_X(A, B)$ which corresponds to the isomorphism class of this family of extensions. 

\begin{proposition} \label{prop:coarse}
$U_{\vec{f}}$ is the coarse space of $\bU_{\vec{f}}$. In particular, since $\bU_{\vec{f}}$ is an algebraic space, the natural map $\bU_{\vec{f}} \to U_{\vec{f}}$ is an isomorphism. 
\end{proposition}
\begin{proof}
Let $\cO(\vec{f}) = Q_1 \oplus \cdots \oplus Q_{m+1}$ where $Q_i$ is perfectly balanced and $\mu(Q_i) < \mu(Q_{i+1})$. We realize $U_{\vec{f}}$ as the top of a tower of affine spaces
\[
  U_{\vec{f}} = U_{\vec{f}, m} \to U_{\vec{f}, m-1} \to \cdots \to U_{\vec{f},1},
\]
and inductively show that 
\begin{itemize}
   \item there is an identification $U_{\vec{f},i} \cong \oplus_{k = 1}^{i} \Ext^1(Q_{k+1}, \oplus_{j = 1}^{k} Q_j)$; 
   \item Let $p: U_{\vec{f}, i} \times \bP^1 \to \bP^1$. Then a $T$-point of $U_{\vec{f}, i}$ is the data of isomorphism classes of extensions
    \[
      0 \to \cK_{k} \to \cK_{k+1} \to p^* Q_{k+1} \to 0
    \]
    on $T \times \bP^1$, where $k \leq i$, $\cK_1 \cong p^* Q_1$. 
\end{itemize}
Let $U_{\vec{f},1} = \Ext^1_{\bP^1}(Q_2, Q_1)$, which satisfies the claim. By induction, assume that the claim is true for $U_{\vec{f},i}$. Abusing notation, let \[
  0 \to \cK_{k} \to \cK_{k+1} \to p^* Q_{k+1} \to 0
\] where $1 \leq k \leq i$ denote the isomorphism classes of extensions on $U_{\vec{f},i} \times \bP^1$ corresponding to the identity map for $U_{\vec{f},i}$. 

Now we let $U_{\vec{f},i+1}$ be the total space of the vector bundle $\cV = R^1\pi_* \cHom(p^* Q_{i+2}, \cK_{i+1})$. Note that by Lemma \ref{lem:ext_iterated_leq} and cohomology and base change, $\cV$ is a vector bundle. A $T$-point of $U_{\vec{f},i+1}$ is the data of a morphism $q: T \to U_{\vec{f}, i}$ and a global section of $q^* \cV = R^1\nye{\pi}_* \nye{q}^* \cHom(p^* Q_{i+2}, \cK_{i+1})$. (The relevant maps are labelled in Figure \ref{fig:Uf-diagram}.)
\begin{figure}
\[
\begin{tikzcd}
T \times \bP^1 \arrow[r, "\nye{q}"] \arrow[d, "\nye{\pi}"] & {U_{\vec{f},i} \times \bP^1} \arrow[r, "p"] \arrow[d, "\pi"] & \bP^1 \\
T \arrow[r, "q"]                                           & {U_{\vec{f},i}}                                              &      
\end{tikzcd}
\]
\caption{Relevant maps for the proof of Proposition \ref{prop:coarse}}
\label{fig:Uf-diagram}
\end{figure}
Since \[R\cHom(\nye{q}^*p^* Q_{i+2}, \nye{q}^*\cK_{i+1}) = \cHom(\nye{q}^*p^* Q_{i+2}, \nye{q}^*\cK_{i+1}),\] and $R \nye{\pi}_* \cHom(\nye{q}^*p^* Q_{i+2}, \nye{q}^*\cK_{i+1})$ is only nonzero for $R^1\nye{\pi}_*$, we have that
\[
  H^0(T, q^* \cV) \cong 
  H^0(T, 
  R^1\nye{\pi}_* \cHom(\nye{q}^*p^* Q_{i+2}, \nye{q}^*\cK_{i+1})
  ) 
  \cong \Ext^1_{T \times \bP^1}(\nye{q}^*p^* Q_{i+2}, \nye{q}^*\cK_{i+1}).
\]
In particular, a section of $q^* \cV$ is indeed an isomorphism class of an extension 
\[
  0 \to \cK_{i+1} \to \cK_{i+2} \to \nye{q}^*p^* Q_{i+2} \to 0
\]
on $T \times \bP^1$. Combined with the interpretation of a $T$-point of $U_{\vec{f},i}$, this proves that the functor of points of $U_{\vec{f},i+1}$ is as claimed by the inductive hypothesis.  

Moreover, since $U_{\vec{f},i}$ is in fact affine space, $\cV$ is trivial and may be noncanonically identified with the fiber at the point of $U_{\vec{f},i}$ that corresponds to the split extensions, where the fiber is $\Ext^1(Q_{i+2}, \oplus_{j = 1}^{i+1} Q_i)$. Hence $U_{\vec{f},i+1}$ is isomorphic to the affine space $\oplus_{k = 1}^{i+1} \Ext^1(Q_{k+1}, \oplus_{j = 1}^{k} Q_j)$. 

Letting $i = m$, the claim shows that the functor of points of $U_{\vec{f}}$ is that of $\bU_{\vec{f}}$ quotiented by automorphisms. In particular, $U_{\vec{f}}$ is the coarse space of $\bU_{\vec{f}}$. However, since the latter is an algebraic space, $\bU_{\vec{f}} \to U_{\vec{f}}$ is an isomorphism by the universal property of coarse spaces.
\end{proof}

\begin{proof}[Proof of $(2) \Rightarrow (6) \Rightarrow (1)$ of Proposition \ref{prop:equiv_leq}]
(6) $\Rightarrow$ (1): Recall that $U_{\vec{f}}$ is just an affine space with $\Sigma_{\vec{f}} = \{0\}$. (6) says that there is a point in $U_{\vec{f}}$ corresponding to the extension of type $\vec{f}$ realizing $\vec{e}$. By considering a line connecting this point to $0$, we find the desired family in (1). 

(2) $\Rightarrow$ (6): This follows from Proposition \ref{prop:moduli_construction} and the discussion about the image of $\bU_{\vec{f}}$. 
\end{proof}

\subsection{Fitting support scheme structure} 
We review in this subsection the Fitting support scheme structure on splitting loci, which is known to be reduced. 

Let $\cF$ be a finitely presented coherent sheaf on a locally Noetherian scheme $X$, and let 
\[
  \cE_1 \xrightarrow{\Phi} \cE_0 \to \cF \to 0
\]
be one such presentation, where $\cE_0, \cE_1$ are locally free of rank $n, m$. Then $\text{Fitt}_i(\cF) \subseteq \cO_X$ is the ideal sheaf generated by the $(n-i)$-minors of $\Phi$. The Fitting ideal $\text{Fitt}_i(\cF)$ is independent of the choice of presentation (e.g. Corollary 20.4 of \cite{Eisenbud1995}). 

The Fitting support scheme strucutre on $\Splitt$ is given by the ideal sheaf $\sum_{m \in \Z} \text{Fitt}_{i_m}(R^1 \pi_* \cE(m))$, where $i_m = h^1 \cO(\vec{e})(m) - 1$ and $\pi: B \times \bP^1 \to B$ is the projection. This is the most common scheme structure used in the literature. Condition (3) of Proposition \ref{prop:equiv_leq} shows that this ideal cuts out the right closed points. The following result shows that the Fitting support scheme structure agrees with the reduced closed subscheme structure.

\begin{theorem}[\cite{LLV}] \label{thm:LLVnormality}The Fitting support scheme structure on splitting loci is normal and Cohen-Macaulay, and in particular reduced. 
\end{theorem}

The Fitting support scheme structure is more amenable to computations. By working out a two-term complex of free modules representing the direct image $R \pi_* \cE(m)$, Corollary 5.4 of \cite{Eisenbud-Schreyer} writes down explicit equations for splitting loci inside $U_{\vec{e}_d}$, where $\vec{e}_d = (0^{r-1}, d)$.
Notably, in the rank $2$ case, splitting loci are the affine cones over secant varieties of the rational normal curve.

\begin{example}[Rank $2$ splitting loci, \cite{Eisenbud-Schreyer}] 
Let $d > 1$ be a positive integer. Let $\Spec A$ be the affine space corresponding to the vector space $\Ext^1_{\bP^1}(\cO(d), \cO)$, with coordinates $a_0, \dots, a_{d-2}$. Let $\cE$ be the universal vector bundle on $\bP^1_A$. Then for $1 \leq k \leq d-1$, the following complex represents $R \pi_* \cE(-k-1)$:
\[
  0 \to A^{d-k} \xrightarrow{B_{k, d-k}} A^{k} \to 0,
\]
where $B_{k, d-k}$ is the $k \times (d - k)$ Hankel matrix
\[
  B_{k, d-k} = 
  \begin{bmatrix}
      a_0 & a_1 & \cdots & a_{d-k-1} \\ 
      a_1 & a_2 & \cdots & a_{d-k} \\ 
      \vdots & \vdots & \vdots & \vdots \\ 
      a_{k-1} & a_k & \cdots & a_{d-2}
  \end{bmatrix}.
\]
Let $\frac{d}{2}+1 \leq e < d$, so that $(d-e, e)$ is an unbalanced splitting type that appears in $\Ext^1(\cO(d), \cO)$ and is not $(0, d)$. Then one can check using the Fitting support definition that the ideal $I_{\overline{\Sigma}_{(d-e, e)}}$ is precisely the ideal of maximal minors $I_{d-e+1}(B_{d-e+1, e-1})$. It is well-known that the latter is the homogeneous ideal of $\Sec^{d-e-1}(C_{d-2})$, the $(d-e-1)$-th secant variety of the rational normal curve of degree $d-2$ in $\bP^{d-2}$ (e.g. \cite{CONCA2018111}). In particular, when $d-e = 1$, $I_{\overline{\Sigma}_{(1,d-1)}} = I_2(B_{2, d-2})$ and $\overline{\Sigma}_{(1,d-1)} \subset \Ext^1(\cO(d), \cO)$ is the cone over the rational normal curve of degree $d-2$. The case of $\overline{\Sigma}_{(1,3)} \subseteq \Ext^1(\cO(4), \cO)$, along with the resolution of singularities we construct, is illustrated in Figure \ref{fig:resolution_sigma13}.

\begin{figure}
  \centering
  \begin{tikzpicture}[scale = 0.7]
    \definecolor{labelColor}{RGB}{112,127,151}
    \definecolor{iris}{RGB}{129, 102, 138}
  % left plot
  \begin{axis}[
    name=left,
    hide axis,
  view={-20}{20},
  domain=-1.5:1.5,
  y domain=0:-2*pi,
  xmin=-1.5, xmax=1.5,
  ymin=-1.5, ymax=1.5, 
  zmin=-1, zmax = 1,
  samples=40,
  samples y=40,
  z buffer=sort,
  width=10cm,
  height=9.5cm,
  ticks=none
  ]
    \addplot3[surf,
    mesh/interior colormap=
       {blueblack}{color=(Lavender) color=(LightSteelBlue)},
    mesh/interior colormap thresh=1,
    shader = faceted interp,
    faceted color = none,
    samples = 30,
    ]
  ({x},{0.66*cos(deg(y))},{0.66*sin(deg(y))});

  \addplot3[variable=t,domain={pi/2}:{3*pi/2},mesh,color=iris] ({0},{0.66*cos(deg(t))},{0.66*sin(deg(t))});
  \addplot3[variable=t,domain={pi/2}:{-pi/2},mesh,dash pattern=on 1pt off 2pt, color=iris] ({0},{0.66*cos(deg(t))},{0.66*sin(deg(t))});

  \draw (0,0,0) 
      node at (-0.3,0.03,0) [iris] {(3) $\hookrightarrow$ (0,4)};

  \draw (0,0,0) 
      node at (0.8,-1, 0.35) [labelColor] {(3) $\hookrightarrow$ (1,3)};
  \end{axis}

  % right plot, anchored at the same vertical centre
  \begin{axis}[
    name=right,
    at={(left.east)}, 
    anchor=west,
    xshift=2cm,
    hide axis,
  view={-20}{20},
  domain=0:1.5,
  y domain=0:-2*pi,
  xmin=-1.5, xmax=1.5,
  ymin=-1.5, ymax=1.5, 
  zmin=-1, zmax = 1,
  samples=40,
  samples y=40,
  z buffer=sort,
  width=10cm,
  height=8cm,
  ticks=none
  ]
    \coordinate (L1) at (-1.5, -1.5, -1);
  \coordinate (L2) at (-1.5, -1.5, 1);
  \coordinate (L3) at (-1.5, 1.5, 1);
  \coordinate (L4) at (-1.5, 1.5, -1);
  \coordinate (R1) at (1.5, -1.5, -1);
  \coordinate (R2) at (1.5, -1.5, 1);
  \coordinate (R3) at (1.5, 1.5, 1);
  \coordinate (R4) at (1.5, 1.5, -1);

  \draw[-] (L3) -- (L4);
  \draw[-] (L4) -- (L1);
  \draw[-] (R1) -- (R2);
  \draw[-] (R2) -- (R3);
  \draw[-] (R3) -- (R4);
  \draw[-] (R4) -- (R1);

  \draw[-] (R3) -- (L3);
  \draw[-] (R4) -- (L4);

  \addplot3[surf,
      mesh/interior colormap=
         {blueblack}{color=(Lavender) color=(LightSteelBlue)},
      mesh/interior colormap thresh=1,
      shader=interp,
      samples = 30,
      ]
    ({-x},{0.66*x*cos(deg(y))},{0.66*x*sin(deg(y))});
  \addplot3[surf, 
      mesh/interior colormap=
         {blueblack}{color=(Lavender) color=(LightSteelBlue)},
      mesh/interior colormap thresh=1,
      shader=interp,
      samples = 30,
      ]
  ({x},{0.66*x*cos(deg(y))},{0.66*x*sin(deg(y))});

  \draw (0,0,0) 
      node at (0,0,0) [fill = iris, circle, minimum size = 0.1em] {}
      node at (0,0.05,0.25) [iris] {(0,4)};

  \draw (-1,-1,0)
      node at (0.8,-1, 0.2) [minimum size = 0.3em, labelColor] {(1,3)};
      % SlateBlue

  \draw (0,0,-1) 
      node at (0,0,-0.8) [minimum size = 0.3em] {(2,2)};

  \draw[-] (L1) -- (L2);
  \draw[-] (L2) -- (L3);

  \draw[-] (R1) -- (L1);
  \draw[-] (R2) -- (L2);
  \end{axis}
  % arrow from the centre of the tube to the centre of the cone
  \draw[-latex,thick]
    (left.east) -- node[midway, below]{$\rho$} +(1.5,0);
\end{tikzpicture}
  \caption{The splitting locus $\overline{\Sigma}_{(1,3)}$ in $\Ext^1(\cO(4), \cO) \cong \mathbb{A}^3$ and its resolution by the relative Quot scheme parametrizing $\cO(3)$-subsheaves}
  \label{fig:resolution_sigma13}
\end{figure}
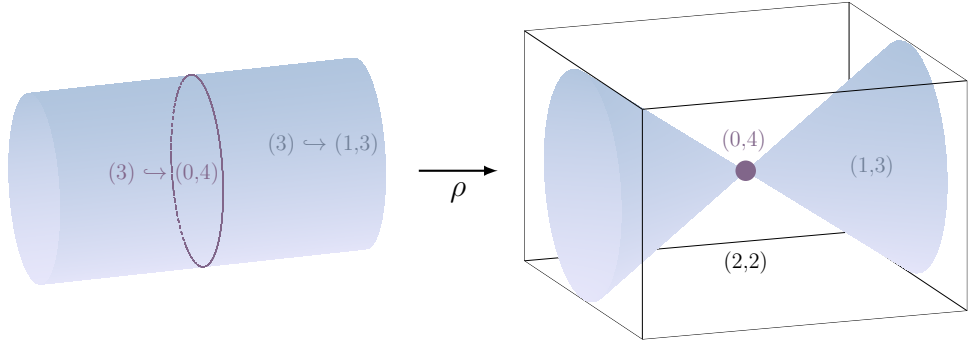
\end{example}
%!TEX root = splitting_loci.tex

\section{A resolution of singularities of \texorpdfstring{$\Splitt$}%
     {splitting loci}} \label{sec:construction_of_the_resolution}
We describe the construction of a resolution of singularities for $\Splitt$. This was initially inspired by the use of Quot schemes to compute the classes of splitting loci in \cite{larson_degeneracy}. 

As explained in the introduction, for any given $\vec{e}$, there may be multiple choices of resolutions. Specifically they are indexed by certain $\vec{e}$-admissible sets $I \subseteq [m]$, where $m$ is the number of subbundles in the Harder-Narasimhan flag of $\cO(\vec{e})$.

\begin{definition} Let $B$ an algebraic stack over $k$, and let $\cE$ be a vector bundle of rank $r$ on $B \times_{k} \bP^1$. Let $\ur = (r_1, \dots, r_m)$, $\ud = (d_1, \dots, d_m)$, where $r \geq r_1 \geq r_2 \geq \cdots \geq r_m \geq 0$, $d_i \in \Z$. We define $\FQuot(B, \cE; \ur, \ud) \to B$ to be the functor
\begin{align*}
  \left (T \to \Spec k \right ) 
  & \mapsto 
  \left \{
    \begin{tabular}{lll}
    \text{A morphism $T \xrightarrow{f} B$ over $k$, and a flag of quotients} \\ 
    \text{$f^* \cE \twoheadrightarrow \cQ_{1}
    \twoheadrightarrow \cQ_{2} 
    \twoheadrightarrow \cdots 
    \twoheadrightarrow \cQ_m$ on $T \times \bP^1$, such that $\cQ_{i}$ is} \\ 
    \text{flat over $T$, and fiberwise of rank $r_i$ and degree $d_i$}
    \end{tabular}
  \right\} / \sim, 
\end{align*}
where we say that two flags are equivalent if there are isomorphisms $\cQ_i \to \cQ_i'$ making the following diagram commute:
\[
\begin{tikzcd}
f^* \cE \arrow[r] \arrow[d, no head, double line] & \cQ_1 \arrow[r] \arrow[d] & \cQ_{2} \arrow[r] \arrow[d] & \cdots \arrow[r] & \cQ_m \arrow[d] \\
f^* \cE \arrow[r]                                & \cQ_1' \arrow[r]          & \cQ_{2}' \arrow[r]          & \cdots \arrow[r] & \cQ_m'         
\end{tikzcd}.
\]
\end{definition}
The abbreviation $\FQuot$ stands for flag Quot, also called hyper-Quot. The map $\FQuot \to B$ is representable by schemes due to the theory of Quot schemes. So if $B$ is an algebraic stack, then for any choice of $\ur, \ud$, $\FQuot(B, \cE; \ur, \ud)$ is an algebraic stack. When $m = 1$ and $\ur = (r_1), \ud = (d_1)$, this is the usual relative Quot, which we denote by $\Quot^{r_1,d_1}(B, \cE)$, or $\Quot^{r_1,d_1}(\cE)$ if the base is clear.

Applying this definition to the pair $\vbst$ and the universal vector bundle $\cE_{\univ}$ on $\vbst \times \bP^1$, we get $\FQuot(\vbst, \cE_{\univ}; \ur, \ud) \to \vbst$, which is the algebraic stack defined by the functor
\begin{align*}
  \left (T \to \Spec k \right ) 
  & \mapsto 
  \left \{
    \begin{tabular}{lll}
    \text{A flag of quotients $\cE \twoheadrightarrow \cQ_{1}
    \twoheadrightarrow \cQ_{2} 
    \twoheadrightarrow \cdots 
    \twoheadrightarrow \cQ_m$ on $T \times \bP^1$, flat over $T$,} \\ 
    \text{such that $\cE$ is a vector bundle of fiberwise rank $r$ and degree $d$, } \\ 
    \text{and $\cQ_{i}$ is a coherent sheaf fiberwise of rank $r_i$ and degree $d_i$}
    \end{tabular}
  \right\} / \sim \\ 
  & \cong 
  \left \{
    \begin{tabular}{lll}
      \text{A flag of subsheaves $\cE_1 \hookrightarrow \cE_2 \hookrightarrow \dots \hookrightarrow \cE_{m+1} = \cE$ on $T \times \bP^1$, flat over $T$,} \\
      \text{such that $\cE$ is a vector bundle of fiberwise rank $r$ and degree $d$, } \\
      \text{and $\cE_{i}$ is locally free of rank $r-r_i$ and fiberwise degree $d-d_i$}
    \end{tabular}
  \right\} / \sim. 
\end{align*}
Here, the equivalence relation on flags of subsheaves is induced by the one on flags of quotients and the correspondence between flags of quotients and subsheaves. Let $\vec{e}$ be an arbitrary splitting type of rank $r$ and degree $d$. Let the Harder-Narasimhan filtration of $\cO(\vec{e})$ be 
\[
  E_1 \subset E_2 \subset \cdots \subset E_{m+1} = \cO(\vec{e}),
\]
and let a corresponding flag of quotients be 
\[
  \cO(\vec{e}) \twoheadrightarrow Q_1 \twoheadrightarrow Q_{2} \twoheadrightarrow \cdots \twoheadrightarrow Q_m.
\]
Let $d_{i,HN} = \deg Q_i$, $r_{i,\HN} = \rk Q_i$, and let $\ur_{\HN} = (r_{1,\HN}, \dots, r_{m,\HN})$, $\ud_{\HN} = (d_{1,\HN}, \dots, d_{m,\HN})$.

\begin{definition} Given a splitting type $\vec{e}$ and $I = \{i_1, \dots, i_k\} \subseteq [m]$, let $\ur_{I,\HN} = (r_{i_1, \HN}, \dots, r_{i_k, \HN})$, and let $\ud_{I,\HN} = (d_{i_1, \HN},\dots, d_{i_k, \HN})$. Let $E_{i_0} = 0, E_{i_{k+1}} = \cO(\vec{e})$. We say that $I$ is $\vec{e}$-admissible if for all $1 \leq j \leq k$, $E_{i_j}/E_{i_{j-1}}$ is balanced.
\end{definition}

\begin{example} \label{ex:e-admissible-sets}
\begin{enumerate}
  \item $I = \emptyset$ is $\vec{e}$-admissible if $\cO(\vec{e})$ is balanced.
  \item If the distinct entries of $\vec{e}$ are all at least $2$ apart, then the only $\vec{e}$-admissible index set is $I = [m]$.
  \item Consider $\vec{e} = (-2,-2,-1,-1,0,0)$. Then $\ur_{\HN} = (4,2), \ud_{\HN} = (-6,-4)$, and the HN flag of subsheaves is 
  \[
    E_1 = \cO(0,0) \hookrightarrow E_2 = \cO(-1,-1,0,0) \hookrightarrow \cO(\vec{e}).
  \]
  The empty set is not $\vec{e}$-admissible, but $\{1\}, \{2\}, \{1,2\}$ are all $\vec{e}$-admissible.
\end{enumerate}
\end{example}

Recall the definition of tame splitting types from the introduction. It's clear that $\vec{e}$ is tame if and only if there exists $I$ which is $\vec{e}$-admissible and $|I| \leq 1$. 

\begin{definition}
We define $\SplittNyeI = \FQuot(\vbst, \cE_{\univ}; \ur_{I, \HN}, \ud_{I, \HN})$, and $\SplittNyeHN = \FQuot(\vbst, \cE_{\univ}; \ur_{\HN}, \ud_{\HN})$.
\end{definition}

\begin{theorem} \label{thm:smoooth}
Let $\vec{e}$ be a splitting type of rank $r$ and degree $d$, and let $I$ be $\vec{e}$-admissible. Then $\SplittNyeI$ is a smooth and irreducible algebraic stack, proper over $\vbst$. The natural map $\SplittNyeI \to \vbst$ is surjective onto $\Splitt$, and is an isomorphism over $\SplittOpen$. In particular, $\dim \SplittNyeI =-r^2 - u(\vec{e})$.
\end{theorem}

The stack $\SplittNyeI$ is proper over $\vbst$ because relative Quot constructions are proper over the base. Let $I$ be $\vec{e}$-admissible. We say that $\cO(\vec{f})$ admits a flag of subbundles of type $(\vec{e}, I)$ if there exists $E_{i_1}' \subset \cdots \subset E_{i_m}' \subset \cO(\vec{f})$ such that $\rk E_{i_j}' = r-r_{i_j}$, and $\deg E_{i_j}' = d-d_{i_j}$. We add one more equivalent condition to the list in Proposition \ref{prop:equiv_leq}.
\begin{proposition} \label{prop:one_more_cond_equiv_leq}
Let $I$ be $\vec{e}$-admissible. Then $\vec{f} \leq \vec{e}$ if and only if $\cO(\vec{f})$ admits a flag of subbundles of type $(\vec{e},I)$. (We call this condition (5').) Moreover, $\cO(\vec{e})$ admits a unique flag of subbundles of type $(\vec{e},I)$.
\end{proposition}
\begin{proof}
The same argument for $(4) \Leftrightarrow (5)$ in the proof of Proposition \ref{prop:equiv_leq} applies here verbatim to show that $(4) \Leftrightarrow (5')$, since Lemma \ref{lem:subsheaf} only requires that the first part of $\vec{e}$ is balanced. 

Let $\vec{e} = (\vec{a}, \vec{b})$, where $\rk \vec{a} = s$ and $a_s < b_1$. Then up to automorphisms of $\cO(\vec{b})$, we claim that there is a unique injection $\cO(\vec{b}) \hookrightarrow \cO(\vec{e})$. This is because after fixing a splitting $\cO(\vec{e}) = \cO(\vec{a}) \oplus \cO(\vec{b})$, it's clear that every nonzero homomorphism $\cO(\vec{b}) \to \cO(\vec{e})$ factors through $\cO(\vec{b})$. So the space of injections from $\cO(\vec{b})$ to $\cO(\vec{e})$ is an $\Aut(\cO(\vec{b}))$-torsor. Now iteratively applying this claim shows that a flag of subbundles of type $(\vec{e}, I)$ of $\cO(\vec{e})$ is unique.
\end{proof}
Assuming irreducibility of $\SplittNyeI$, we have $(2) \Leftrightarrow (5) \Leftrightarrow (5')$ from Proposition \ref{prop:equiv_leq}. This implies that the natural map $\SplittNyeI \to \vbst$ is surjective onto $\Splitt$. The uniquess statement from Proposition \ref{prop:one_more_cond_equiv_leq} shows that this map is an isomorphism over $\SplittOpen$. Thus to prove Theorem \ref{thm:smoooth}, it remains to show irreducibility and smoothness. 

\subsection{Irreducibility}
We prove irreducibility for any hyper-Quot construction on $(\vbst, \cE_{\univ})$. The main idea is if $Z \subseteq X$ is a union of irreducible strata, such that only one stratum has codimension equal to the expected codimension of $Z$ in $X$, then $Z$ must be irreducible of the expected codimension and that particular stratum is dense in $Z$.

\begin{definition}
Let $B$ be an irreducible scheme with a map $\phi_{\cE}: B \to \vbst$, corresponding to $\cE$ on $B \times \bP^1$. We say that the pair $(B, \cE)$ is good, or the map $\phi_{\cE}$ is good, if each splitting locus $\Sigma_{\vec{e}}(B)$ is irreducible of codimension $u(\vec{e})$.
\end{definition}

\begin{remark}
Note that $\phi_\cE$ may not be good even if it is smooth. Smoothness guarantees that the codimension of $\Sigma_{\vec{e}}$ is as expected, but $\Sigma_{\vec{e}}$ may become reducible after base change along a smooth map. However, note that the smooth covers $U_{\vec{f}} \to \vbst$ from Section \ref{subsec:sm_cover} are good. Indeed, the description of $U_{\vec{f}}$ in Proposition \ref{prop:moduli_construction} implies that the preimage of $\mathbf{\Sigma}_{\vec{e}}$ is irreducible. Note that since the codimension of $\splitt$ is at most $u(\vec{e})$, when $\phi_\cE$ is good, the closure of $\Sigma_{\vec{e}}$ is $\splitt$.
\end{remark}

Let $B$ be a scheme of finite type with a good map $\phi_{\cE}: B \to \vbst$. Let $0 \leq s \leq r$ and $d' \in \Z$. 
For this section, we are mostly thinking about $\Quot$ as parametrizing families of subsheaves with flat quotients. So it will be convenient to write $S_{s,d'}(B, \cE)$ to denote $\Quot^{r-s,d-d'}(B,\cE)$. 
When the base is clear, we abbreviate them as $S_{s,d'}(\cE)$ and $\Quot^{r-s,d-d'}(\cE)$.
Since $B$ is finite type, there exists $c \in \Z$, $m \in \Z$ such that $\cE^{\vee}(c)$ and $\cE(m)$ are relatively globally generated; equivalently, $R^1 \pi_* \cE^{\vee}(c-1) = 0$ and $R^1 \pi_* \cE(m-1) = 0$. For ease of notation, let $\cF = \pi_* \cE^{\vee}(c)$ and let $\cG = \pi_* \cE^{\vee}(c-1)$. Let $\cF' = \pi_* \cE(m)$, $\cG' = \pi_* \cE(m-1)$. Note that these are all vector bundles by cohomology and base change. The short exact sequence of Str{\o}mme (Proposition 1.1 \cite{Stromme1987}), applied to $\cE^{\vee}(c)$ and $\cE(m)$, after possibly dualizing and twisting, gives us
\begin{equation} \label{eq:SES1}
  0 \to \cE \to \pi^* \cF(c) \to \pi^*\cG(c+1) \to 0,
\end{equation}
and
\begin{equation} \label{eq:SES2}
  0 \to \pi^* \cG'(-m-1) \to \pi^* \cF'(-m) \to \cE \to 0.
\end{equation}
Using \eqref{eq:SES1}, by postcomposing, we can think of subsheaves of $\cE$ with flat quotients as subsheaves of $\pi^*\cF(c)$ with flat quotients. Similarly, using \eqref{eq:SES2}, we can think of quotients of $\cE$ as quotients of $\pi^* \cF'(-m)$ by precomposing. This gives rise to two different embeddings, as in Figure \ref{fig:two_embs}. For the discussion of irreducibility, we will only need the embedding $\iota$. We label the relevant maps in Figure \ref{fig:pullback}.

\begin{figure}
\centering
\begin{minipage}{.5\textwidth}
  \centering
  \[
  \begin{tikzcd}
{S_{s,d'}(\cE)} \arrow[rd] \arrow[r, "\iota", hook] & {S_{s,d'}(\pi^* \cF(c))} \arrow[d] \\
& B \end{tikzcd}
\]

\end{minipage}%
\begin{minipage}{.5\textwidth}
  \centering
  \[
  \begin{tikzcd}
{\Quot^{r-s,d-d'}( \cE)} \arrow[rd] \arrow[r, "j", hook] & {\Quot^{r-s,d-d'}(\pi^* \cF'(-m))} \arrow[d] \\
& B 
\end{tikzcd}
\]
  % \label{fig:emb2}
\end{minipage}
\caption{Two embeddings of a relative Quot scheme into a fiber bundle of Quot schemes}
\label{fig:two_embs}
\end{figure}

\begin{figure}
\[
\begin{tikzcd}
S_{s,d'}(\cE) \times \bP^1 \arrow[r, hook] \arrow[d] & 
S_{s,d'}(\pi^* \cF(c)) \times \bP^1 \arrow[r, "\nye{\gamma}"] \arrow[d, "f"] 
& B \times \bP^1 \arrow[d, "\pi"] \\
S_{s,d'}(\cE) \arrow[r, hook, "\iota"] \arrow[rr, bend right, "\sigma"]                        & 
S_{s,d'}(\pi^* \cF(c)) \arrow[r, "\gamma"]         & B                              
\end{tikzcd}
\]

\[
\begin{tikzcd}
\Quot^{r-s,d-d'}(\cE) \times \bP^1 \arrow[r, hook] \arrow[d] & 
\Quot^{r-s,d-d'}(\pi^* \cF'(-m)) \times \bP^1 \arrow[r, "\nye{\gamma}"] \arrow[d, "f"] 
& B \times \bP^1 \arrow[d, "\pi"] \\
\Quot^{r-s,d-d'}(\cE) \arrow[r, hook, "j"]                        & 
\Quot^{r-s,d-d'}(\pi^* \cF'(-m)) \arrow[r, "\gamma"]         & B                              
\end{tikzcd}
\]
\caption{The relative Quot scheme and relavant maps. We abuse notation and use the same symbols for the two scenarios.}
\label{fig:pullback}
\end{figure}

A subsheaf of $\pi^* \cF(c)$ is a subsheaf of $\cE$ if and only if its image in $\pi^* \cG(c+1)$ is zero. So the image of the embedding $\iota$ in Figure \ref{fig:two_embs} is closed, and is the zero locus of the section of the vector bundle $f_* \cHom(\cS, \nye{\gamma}^* \pi^* \cG(c+1))$ corresponding to the composite $\cS \to \nye{\gamma}^* \pi^* \cF(c) \to \nye{\gamma}^* \pi^* \cG(c+1)$, where $\cS$ is the universal subsheaf on $S_{s,d'}(\pi^*\cF(c)) \times \bP^1$.
Dually, the image of the closed embedding $j$ in Figure \ref{fig:two_embs} is the zero locus of the section of the vector bundle $f_* \cHom(\nye{\gamma}^* \pi^* \cG'(-m-1), \cQ)$, corresponding to the composite $\nye{\gamma}^* \pi^* \cG'(-m-1) \to \nye{\gamma}^* \pi^* \cF'(-m) \to \cQ$, where $\cQ$ is the universal quotient on $\Quot^{r-s,d-d'}(\pi^*\cF'(-m)) \times \bP^1$. Our calculations below will show, 
among other things, that these sections cut out $S_{s,d'}(\cE)$ and $\Quot^{r-s',d-d'}(\cE)$ as local complete intersections in their embeddings.

%!TEX root = splitting_loci.tex
\begin{lemma} \label{lem:eb} Given $\vec{a}, r, d$, there is a unique splitting type $\vec{e}_b = \vec{e}_b(r,d,\vec{a})$ admitting a $\cO(\vec{a})$-subsheaf which can be written as $(\text{Bal}, \vec{a}_+)$, where $\text{Bal}$ is balanced, and $\vec{a} = (\vec{a}_-, \vec{a}_+)$. This splitting type $\vec{e}_b$ is maximal among all $\vec{e}$ of rank $r$ and degree $d$ such that $\cO(\vec{e})$ has a $\cO(\vec{a})$-subsheaf. Moreover, if $\pvec{a}' \leq \vec{a}$, then $\vec{e}_b(r,d,\pvec{a}') \leq \vec{e}_b(r,d,\vec{a})$. 
\end{lemma}

\begin{remark} For instance, if $\vec{a} = (0,0)$, and $(r, d) = (4,-1)$, then $\vec{e}_b = (-1,0,0,0)$. Notice that in this case, we can take $\vec{a}_+$ to be length $0,1$ or $2$. The claim of the lemma is that $\vec{e}_b$ is unique, but the decomposition into two parts may not be.
\end{remark}

\begin{proof}[Proof of Lemma \ref{lem:eb}]
Let $A_k = \sum_{j = 0}^{k-1}a_{s-j}$. To form such an $\vec{e}_b$ amounts to choosing the length $k$ of $\vec{a}_+$, and letting the first part be the balanced splitting type of the correct degree. The choice of a particular length would not work if the vector obtained is not a weakly increasing sequence, i.e. if $\ceil{\frac{d-A_k}{r-k}} > a_{s-k+1}$. It also would not work if the vector obtained, as a splitting type, corresponds to a vector bundle that does not admit a $\cO(\vec{a})$-subsheaf. Recall that by Lemma \ref{lem:subsheaf}, $\cO(\vec{e}_b)$ admits an $\cO(\vec{a})$-subsheaf if and only if $a_{s-j} \leq e_{b,r-j}$ for all $0 \leq j < s$.

There are two ways the lemma may be false. One is that no $\vec{a}_+$ works. We claim that in this case $\frac{d}{r} \geq a_{s-k+1}$ for all $1 \leq k \leq s$, so that it suffices to choose $k = 0$. 
Indeed, $k = s$ doesn't work if and only if $\ceil{\frac{d-A_s}{r-s}} > a_{1}$,
which implies that $
\frac{d}{r} \geq a_{1}.$
Now since \[
\ceil{\frac{d-A_{s-1}}{r-s+1}} = \ceil{\frac{d-A_{s}+a_1}{r-s+1}} > a_{1},
\] $k = s-1$ doesn't work if and only if $
\ceil{\frac{d-A_{s-1}}{r-s+1}} > a_{2}.
$
The desired claim follows from induction like so.

Another way that this claim could fail is that there are two valid choices of $k$. Suppose that $k_1 > k_2$ both work, giving splitting types $\vec{e}_b$ and $\pvec{e}_b'$. Since $k_1 > k_2$, $e_{b, r-k_2}' \leq e_{b, r-k_2}$. But in order for $\pvec{e}_b'$ to admit an $\cO(\vec{a})$-subsheaf, we must have $e_{b, r-k_2}' \geq a_{s-k_2} = e_{b, r-k_2}$. Therefore $e_{b, r-k_2}' = e_{b, r-k_2}$, which means that $\vec{e}_b$ must have been balanced up to the index $r-k_2$ to begin with. So in fact $\vec{e}_b = \pvec{e}_b'$.

For the next claim, write $\vec{e}_b = (\text{Bal}, \vec{a}_+)$, and let $k = \len(\vec{a}_+)$. (As explained in the remark above, there may be multiple ways of breaking $\vec{e}_b = (e_{b,1}, \dots, e_{b,r})$ up into two parts, but just choose one way.)
Let $\vec{e}$ be some splitting type of rank $r$ and degree $d$ such that $\cO(\vec{e})$ has a $\cO(\vec{a})$-subsheaf. Then since $a_{s-j} \leq e_{r-j}$ for $0 \leq j < k$, we have for all $1 \leq i \leq k$,
\[
\sum_{j = 0}^{i-1}e_{b, r-j} = 
\sum_{j = 0}^{i-1}a_{s-j} 
\leq 
\sum_{j = 0}^{i-1}e_{r-j}.
\]
Since the first part of $\vec{e}_b$ is balanced with $\deg \text{Bal} \geq \sum_{j = 1}^{r-k}e_j'$, we must have for $1 \leq i \leq r-k$,
\[
\sum_{j = 1}^{i}e_{b, j} \geq  
\sum_{j = 1}^{i}e_{j}. 
\]
Together these inequalities show that $\vec{e}_b \geq \vec{e}$. 

The second comparison follows along similar lines. Write $\pvec{e}_b' = \vec{e}_b(r,d,\pvec{a}') = (\text{Bal}', \pvec{a}_+')$ and let $k' = \len(\pvec{a}_+')$. We divide into two cases: $k \geq k'$ and $k < k'$. First suppose $k \geq k'$. For $1 \leq i \leq k$, we have 
\begin{align} \label{eq:compare}
\sum_{j = 0}^{i-1}e_{b,r-j} = 
\sum_{j = 0}^{i-1}a_{s-j} 
\leq 
\sum_{j = 0}^{i-1}a_{s-j}' 
\leq
\sum_{j = 0}^{i-1}e'_{b,r-j}.
\end{align}
It remains to compare $\text{Bal}$ and the first $r-k$ entries of $\text{Bal}'$. Both are balanced, with $\deg \text{Bal} \geq \deg \text{Bal}'|_{1,\dots,r-k}$. So for $1 \leq i \leq r-k$,
\[
\sum_{j = 1}^{i}e_{b, j} \geq  
\sum_{j = 1}^{i}e'_{b, j}. 
\]
Therefore 
$\vec{e}_b(r,d,\pvec{a}') \leq \vec{e}_b(r,d,\vec{a})$.

Now suppose $k < k'$. For $1 \leq i \leq k$, Equation \ref{eq:compare} still holds, with the last inequality being in fact an equality in this case. Suppose that for some $k < i \leq k'$, 
\begin{align} \label{eq:compare2}
\sum_{j = 0}^{i-1}e_{b,r-j} = |\vec{a}_+| + \sum_{j = k}^{r-1}e_{b,r-j} >
\sum_{j = 0}^{i-1}e'_{b,r-j} = \sum_{j = 0}^{i-1}a'_{r-j}.
\end{align}
Since the inequality comparing the partial sums from the back changed as we went from summing the last $k$ terms to summing the last $i$ terms, it must be true that 
\[
e_{b, r-k} = \ceil{\frac{d-|\vec{a}_+|}{r-k}}
> a'_{r-i+1} \geq e'_{b,r-k'} = \ceil{\frac{d-|\pvec{a}_+'|}{r-k'}}.
\]
In particular, $e_{b, j} \geq e_{b, j}'$ for all $1 \leq j \leq r-i$. However, that's impossible because together with Equation \ref{eq:compare2}, this implies that $d = |\vec{e}_b| > |\pvec{e}_b'| = d$. This shows that for $1 \leq i \leq k'$, 
\[
\sum_{j = 0}^{i-1}e_{b,r-j}
\leq
\sum_{j = 0}^{i-1}e'_{b,r-j}.
\]
Now since the first $r-k'$ entries of $\vec{e}_b$ and $\pvec{e}_b'$ are balanced with that of $\vec{e}_b$ having degree no less than that of $\pvec{e}_b'$, we can conclude that $\vec{e}_b \geq \pvec{e}_b'$.
\end{proof}

\begin{lemma} \label{lem:ue-ext-inequality} Let $(\vec{a}, \vec{e})$ be such that $\cO(\vec{e})$ admits a $\cO(\vec{a})$-subsheaf. Then
\[
u(\vec{e}) - \operatorname{ext}^1(\cO(\vec{a}), \cO(\vec{e})) \geq 0,
\]
with equality if and only if $\vec{e} = \vec{e}_b(r,d,\vec{a})$. 
\end{lemma}
\begin{proof} 
Write $\vec{e} = (e_1, \dots, e_r)$ and $\vec{a} = (a_1, \dots, a_s)$. We can write
\begin{align*}
u(\vec{e}) - \operatorname{ext}^1(\cO(\vec{a}), \cO(\vec{e}))
&= 
\operatorname{ext}^1(\cO(e_1, \dots, e_{r-s}), \cO(\vec{e}))
+ 
\sum_{i = 0}^{s-1} (
h^1(\cO(\vec{e}) \otimes \cO(e_{r-i})^{\vee})
- h^1(\cO(\vec{e}) \otimes \cO(a_{s-i})^{\vee})) \\
&= u(e_1, \dots, e_{r-s}) + \sum_{i = 0}^{s-1} (
h^1(\cO(\vec{e}) \otimes \cO(e_{r-i})^{\vee})
- h^1(\cO(\vec{e}) \otimes \cO(a_{s-i})^{\vee})).
\end{align*}
Since $\cO(\vec{e})$ has a $\cO(\vec{a})$-subsheaf, we must have that $a_{s-i} \leq e_{r-i}$ for $0 \leq i < s$, so each difference in the sum above is nonnegative. It is equal to zero if and only if $u(e_1, \dots, e_{r-s}) = 0$, and 
\[
h^1(\cO(\vec{e}) \otimes \cO(e_{r-i})^{\vee})
= h^1(\cO(\vec{e}) \otimes \cO(a_{s-i})^{\vee})
\]
for all $0 \leq i < s$. The latter condition says that whenever there exists $1 \leq j \leq r$ such that $e_{r-i} - e_j \geq 2$, then $a_{s-i} = e_{r-i}$. In other words, $\vec{e}$ achieves equality if and only if $\vec{e}$ can be written as $(\text{Bal}, \vec{a}_+)$, where $\text{Bal}$ is balanced, and $\vec{a} = (\vec{a}_-, \vec{a}_+)$. 
\end{proof}

\begin{definition}
Let $\cS$ be the universal subsheaf on $S_{s,d'}(B,\cE) \times \bP^1$. Let $\sigma: S_{s,d'}(B,\cE) \to B$. Define $X_{\vec{a},\vec{e}} = \Sigma_{\vec{a}}(\cS) \cap \sigma^{-1}\Sigma_{\vec{e}}(\cE)$. Define $\Subsheaves(\vec{a},\vec{e}) = U(\vec{a},\vec{e}) / \Aut(\cO(\vec{a}))$, where $U(\vec{a}, \vec{e}) \subset \Hom(\cO(\vec{a}), \cO(\vec{e}))$ is the open locus of injective maps, and $\Aut(\cO(\vec{a}))$ acts on $U_{\vec{a},\vec{e}}$ by precomposition.
\end{definition}

\begin{lemma} \label{lemma:subsheaves}
The map $X_{\vec{a},\vec{e}} \to \Sigma_{\vec{e}}(\cE)$ is flat with fibers isomorphic to $\Subsheaves(\vec{a},\vec{e})$.
\end{lemma}
\begin{proof}
By taking a base change along a principal $\Aut(\cO(\vec{e}))$-bundle, which is a faithfully flat morphism, we may without loss of generality assume that $\cE \cong g^* \cO(\vec{e})$, where $g: \Sigma_{\vec{e}}(\cE) \times \bP^1 \to \bP^1$. Then we claim that 
$X_{\vec{a},\vec{e}} \cong \Sigma_{\vec{e}}(\cE) \times \Subsheaves(\vec{a},\vec{e})$. 
By functoriality of $\Quot$, 
$
S_{s,d'}(B, g^*\cO(\vec{e}))
\cong 
\Sigma_{\vec{e}}(\cE) 
\times
S_{s,d'}(\Spec k, \cO(\vec{e}))
$. 
We can then take the splitting locus $\Sigma_{\vec{a}}(\cS)$ on both sides. So it remains to show that over a point, the splitting locus $\Sigma_{\vec{a}}(\cS) \subset S_{s,d'}(\Spec k, \cO(\vec{e}))$ is isomorphic to $\Subsheaves(\vec{a},\vec{e})$. The only difference between the functors they represent is that a priori, $\Subsheaves(\vec{a},\vec{e})$ does not require flatness for the family of quotients. So it remains to show that if
\[
	0 \to \cS \to g^* \cO(\vec{e}) \to \cQ \to 0
\]
is the universal short exact sequence on $\Hom(\cO(\vec{a}),\cO(\vec{e})) \times \bP^1$, then $\cQ$ is flat over exactly $U(\vec{a},\vec{e})$. To see this, we use the local criterion for flatness (Theorem 6.8 \cite{Eisenbud1995}), which says that $\cQ$ is flat at a closed point $\phi \in \Hom(\cO(\vec{a}),\cO(\vec{e}))$ if and only if $\operatorname{Tor}^1(\cO_{\bP^1_{\phi}}, \cQ) = 0$. Tensoring the short exact sequence above with $\cO_{\bP^1_{\phi}}$, we see that since 
\[
\operatorname{Tor}^1(\cO_{\bP^1_{\phi}}, g^* \cO(\vec{e})) = 0
\] for all $\phi$, $\operatorname{Tor}^1(\cO_{\bP^1_{\phi}}, \cQ) = 0$ if and only if the map $\phi: \cO(\vec{a}) \to \cO(\vec{e})$ is injective.
\end{proof}

\begin{remark}
The same argument of taking a faithfully flat base change along a frame bundle shows that $S_{s,d'}(\pi^* \cF(c)) \to B$ is flat with fibers isomorphic to the absolute Quot scheme $S_{s,d'}(E)$, where $E = \cO_{\bP^1}(c)^{\oplus \rk \cF}$.
\end{remark}

The argument below follows the ideas of Lemma 6.1 of \cite{larson_degeneracy}.
\begin{proposition} \label{prop:codim_count} 
Suppose that $\phi_{\cE}: B \to \vbst$ is good. 
For each $\vec{a}$, let $\vec{e}_b = \vec{e}_b(r, d, \vec{a})$. 
Let $c$ be such that $\cE^{\vee}(c)$ is globally generated and take the embedding of $S_{s,d'}(\cE) \hookrightarrow S_{s,d'}(\pi^* \cF(c))$, defined above. 
Then $\codim (X_{\vec{a}, \vec{e}_b} \subseteq S_{s,d'}(\pi^* \cF(c))) = u(\vec{a}) + \rk(f_* \cHom(\cS, \nye{\gamma}^* \pi^* \cG(c+1)))$. For all other tuples $(\pvec{a}', \pvec{e}')$ where $\pvec{a}' \leq \vec{a}$ and $\cO(\pvec{e}')$ admits a $\cO(\pvec{a}')$-subsheaf, we have $\pvec{e}' \leq \vec{e}_b$ and $\codim (X_{\pvec{a}', \pvec{e}'} \subseteq S_{s,d'}(\pi^* \cF(c)) > u(\vec{a}) + \rk(f_* \cHom(\cS, \nye{\gamma}^* \pi^* \cG(c+1)))$.
\end{proposition}

\begin{proof}
Let $\vec{a}$, $\vec{e}$ be such that $\cO(\vec{e})$ has an $\cO(\vec{a})$-subsheaf. In particular $a_s \leq e_r \leq c$. Let $\gamma: S_{s,d'}(\pi^*\cF(c)) \to B$. By Lemma \ref{lemma:subsheaves},
\begin{align*}
\codim(X_{\vec{a},\vec{e}} \subseteq S_{s,d'}(\pi^*\cF(c))) &=
u(\vec{e}) + \text{fiberdim}(\gamma)-\dim(\text{Subsheaves}(\cO(\vec{a}), \cO(\vec{e}))) 
\end{align*}
Now, by restricting the exact sequence
\[
0 \to \cE \to \pi^* \cF(c) \to \pi^* \cG(c+1) \to 0
\]
to some $p \in \Sigma_{\vec{e}}(\cE)$, we get that $\cO(\vec{e})$ is a kernel like so:
\begin{equation} \label{eq:SES3}
0 \to \cO(\vec{e}) \to \cO(c)^{\oplus \rk \cF} \to \cO(c+1)^{\oplus \rk \cG} \to 0.
\end{equation}
Since $\text{fiberdim} (\gamma) = \dim S_{s,d'}(\cO(c)^{\oplus \rk \cF})$, and the subsheaf generically splits as the balanced bundle $\cO(\vec{a}_b)$, we can write 
\[
\text{fiberdim} (\gamma)
= 
h^0 \cHom(\cO(\vec{a}_b), \cO(c)^{\oplus \rk \cF})
- 
h^0 \cEnd(\cO(\vec{a}_b)).
\]
Since $c \geq a_s$, 
\[
h^0 \cHom(\cO(\vec{a}_b), \cO(c)^{\oplus \rk \cF})
= 
h^0 \cHom(\cO(\vec{a}), \cO(c)^{\oplus \rk \cF}).
\]
Similarly, 
\[
\dim(\text{Subsheaves}(\cO(\vec{a}), \cO(\vec{e}))) 
= 
h^0 \cHom(\cO(\vec{a}), \cO(\vec{e}))
- 
h^0 \cEnd(\cO(\vec{a})).
\]
Note that $\cEnd(\cO(\vec{a}))$ and $\cEnd(\cO(\vec{a}_b))$ have the same rank and degree, and therefore the same Euler characteristic. Therefore, 
\[
h^0(\cEnd(\cO(\vec{a}))) - h^0(\cEnd(\cO(\vec{a}_b)))
= h^1(\cEnd(\cO(\vec{a}))) - h^1(\cEnd(\cO(\vec{a}_b)))
= u(\vec{a}).
\]
Moreover, we obtain a long exact sequence in $\Ext$ groups by applying $\Hom(\cO(\vec{a}), -)$ to Equation \ref{eq:SES3}. Again since $c \geq a_s$, $\operatorname{ext}^1(\cO(\vec{a}), \cO(c)^{\oplus \rk \cF}) = 0$. So we get that
\begin{align*}
h^0 \cHom(\cO(\vec{a}), \cO(c)^{\oplus \rk \cF})
- h^0 \cHom(\cO(\vec{a}), \cO(\vec{e})) &=
h^0\cHom(\cO(\vec{a}), \cO(c+1)^{\oplus \rk \cG}) - \operatorname{ext}^1(\cO(\vec{a}), \cO(\vec{e})) \\ 
&= \rk f_* \cHom(\cS, \nye{\gamma}^* \pi^* \cG(c+1)) - \operatorname{ext}^1(\cO(\vec{a}), \cO(\vec{e})).
\end{align*}
Combining everything, we get that
\begin{align*}
\codim(X_{\vec{a},\vec{e}} \subseteq S_{s,d'}(\pi^*\cF(c))) &=
u(\vec{a}) + \rk f_* \cHom(\cS, \nye{\gamma}^* \pi^* \cG(c+1)) + u(\vec{e}) - \operatorname{ext}^1(\cO(\vec{a}), \cO(\vec{e})).
\end{align*}

By Lemma \ref{lem:ue-ext-inequality}, for a fixed $\vec{a}$, if $(\pvec{a}', \pvec{e}')$ is such that $\pvec{a}' \leq \vec{a}$, and $\pvec{e}'$ admits a $\cO(\pvec{a}')$-subsheaf, then
\begin{align*} \codim(X_{\pvec{a}',\pvec{e}'} \subseteq S_{s,d'}(B,\cE)) &
\geq u(\pvec{a}') + \rk f_* \cHom(\cS, \nye{\gamma}^* \pi^* \cG(c+1)) \\
& \geq u(\vec{a}) + \rk f_* \cHom(\cS, \nye{\gamma}^* \pi^* \cG(c+1))
\end{align*}
with equality if and only if $\pvec{a}' = \vec{a}$ and $\pvec{e}' = \vec{e}_b(r, d, \vec{a})$. By Lemma \ref{lem:eb}, $\pvec{e}' \leq \vec{e}_b(r, d, \pvec{a}') \leq \vec{e}_b(r, d, \vec{a})$.
\end{proof}

\begin{proposition} \label{prop:irred_onestep}
Let $B$ be an irreducible scheme of finite type with a good map $\phi_{\cE}: B \to \vbst$. Then $S_{s,d'}(\cE)$ is irreducible and the induced map $\phi_{\cS}: S_{s,d'}(\cE) \to \cB_{s,d'}$ is good.
% , and $\overline{\Sigma}_{\vec{a}}(\cS)$ is irreducible of codimension $u(\vec{a})$ in $S_{s,d'}(\cE)$ for all $\vec{a}$, $r$ and $d$. 
The image of $\overline{\Sigma}_{\vec{a}}(\cS)$ in $B$ is exactly $\overline{\Sigma}_{\vec{e}_b}$, where $\vec{e}_b = \vec{e}_b(r,d,\vec{a})$. 
\end{proposition}
\begin{proof}
Recall that $S_{s,d'}(\cE)$ is the vanishing locus of a section of the vector bundle 
\[f_* \cHom(\cS, \nye{\gamma}^* \pi^* \cG(c+1)),\] 
and that $\overline{\Sigma}_{\vec{a}}(\cS)$ is a splitting locus. Therefore the lower bound on codimension given in the previous proposition is also an upper bound on the codimension of a component of $\overline{\Sigma}_{\vec{a}}(\cS)$ in $S_{s,d'}(\pi^* \cF(c))$. Since $\phi_{\cE}$ is good, $\Sigma_{\vec{e}}(\cE)$ is irreducible, which implies that $X_{\vec{a},\vec{e}}$ is irreducible. Thus Proposition \ref{prop:codim_count} implies that $\overline{\Sigma}_{\vec{a}}(\cS)$ is irreducible of codimension $u(\vec{a})$ in $S_{s,d'}(\cE)$ and $X_{\vec{a},\vec{e}_b}$ is dense in $\overline{\Sigma}_{\vec{a}}(\cS)$. Specializing to the case where $\vec{a}$ is balanced proves that $S_{s,d'}(\cE)$ is irreducible. 

By Proposition \ref{prop:codim_count}, the image of $\overline{\Sigma}_{\vec{a}}(\cS)$ in $B$ is closed, contained in $\overline{\Sigma}_{\vec{e}_b}$, and contains $\Sigma_{\vec{e}_b}$. By the assumption that $\phi_{\cE}$ is good, the closure of $\Sigma_{\vec{e}_b}$ is exactly $\overline{\Sigma}_{\vec{e}_b}$. So the image of $\overline{\Sigma}_{\vec{a}}(\cS)$ must be exactly $\overline{\Sigma}_{\vec{e}_b}$.
\end{proof} 

\begin{proposition} \label{prop:lci}
When $\phi_\cE: B \to \vbst$ is good, the embeddings of Figure \ref{fig:two_embs} defined earlier have codimension equal to $\rk \cR$, where
$\cR = f_* \cHom(\cS, \nye{\gamma}^* \pi^* \cG(c+1))$ in the first case, and $\cR = f_* \cHom(\nye{\gamma}^* \pi^* \cG'(-m-1), \cQ)$ in the second case. In particular, $S_{s,d'}(\cE) = \Quot^{r-s,d-d'}(\cE)$ sits as a local complete intersection under these two embeddings.
\end{proposition}
\begin{proof} 
The proposition for the first case follows from the stratification of $S_{s,d'}(B,\cE)$ into irreducible strata $X_{\vec{a},\vec{e}}$ and our computation in Proposition \ref{prop:codim_count}, which says that the strata always have codimension at least $\rk \cR$, and is equal to $\rk \cR$ if and only if $\vec{a}$ is balanced and $\vec{e} = \vec{e}_b(r,d,\vec{a})$. 

To see that the second embedding also has the expected codimension $\rk \cR$, now it suffices to compute the difference in dimension between $X_{\vec{a}_b, \vec{e}_b}$ and $\Quot^{r-s,d-d'}(\pi^*\cF'(-m))$, where $\vec{a}_b$ is balanced and $\vec{e}_b = \vec{e}_b(r,d,\vec{a}_b)$. Let us choose a $k$-point $\alpha \in X_{\vec{a}_b, \vec{e}_b} \subset \Quot^{r-s,d-d'}(\pi^*\cF'(-m))$. Then $\alpha$ viewed as a point of $\Quot^{r-s,d-d'}(\pi^*\cF'(-m))$ corresponds to an exact sequence
\[
	0 \to S \to E \to Q \to 0,
\]
on $\bP^1_k$, where $E \cong \cO(-m)^{\oplus N}$, $N = h^0 \cO(\vec{e}_b)(m)$, and $Q$ is a coherent sheaf of rank $r-s$, degree $d-d'$. Viewed as a point of $X_{\vec{a}_b, \vec{e}_b}$, it corresponds to an exact sequence
\[
	0 \to \cO(\vec{a}_b) \to \cO(\vec{e}_b) \to Q \to 0.
\]
Let $K = \cO(-m-1)^{\oplus M}$, where $M = h^0\cO(\vec{e}_b)(m-1)$. Then these two exact sequences fit in a diagram as in Figure \ref{fig:twoExactSequences}.
\begin{figure}
\[
	\begin{tikzcd}
            & 0 \arrow[d]                                & 0 \arrow[d]                        &                                            &   \\
            & K \arrow[r, double line, no head] \arrow[d] & K \arrow[d]                        &                                            &   \\
0 \arrow[r] & S \arrow[d] \arrow[r]                      & E \arrow[d] \arrow[r]              & Q \arrow[d, double line, no head] \arrow[r] & 0 \\
0 \arrow[r] & \cO(\vec{a}_b) \arrow[r] \arrow[d]         & \cO(\vec{e}_b) \arrow[r] \arrow[d] & Q \arrow[r]                                & 0 \\
            & 0                                          & 0                                  &                                            &  
\end{tikzcd}
\]
\caption{Relationship between the two exact sequences that correspond to the same point $\alpha \in X_{\vec{a}_b, \vec{e}_b}$}
\label{fig:twoExactSequences}
\end{figure}

By deformation theory, the fiber dimension of the map $\gamma: \Quot^{r-s,d-d'}(\pi^*\cF'(-m)) \to B$ is $\chi(\cHom(S, Q))$. But
\begin{align*}
\chi(\cHom(S, Q)) &= \chi(\cHom(K,Q)) + \chi(\cHom(\cO(\vec{a}_b), Q)) \\ 
&= \dim \Hom(K, Q) + \chi(\cHom(\cO(\vec{a}_b), \cO(\vec{e}_b))) - \chi(\cEnd(\cO(\vec{a}_b))) \\ 
&= \dim \Hom(K, Q) + \dim \Hom(\cO(\vec{a}_b), \cO(\vec{e}_b)) - \operatorname{ext}^1(\cO(\vec{a}_b), \cO(\vec{e}_b)) - h^0 \cEnd(\cO(\vec{a}_b)) \\ 
&= \dim \Hom(K, Q) - u(\vec{e}_b) + \dim \Hom(\cO(\vec{a}_b), \cO(\vec{e}_b)) - h^0 \cEnd(\cO(\vec{a}_b)), 
\end{align*}
where the last equality is due to Lemma \ref{lem:ue-ext-inequality}.
It follows now that
\begin{align*}
	\rk \cR &= \dim \Hom(K, Q) \\
	&= \operatorname{fiberdim}(\gamma) + u(\vec{e}_b) - (\dim \Hom(\cO(\vec{a}_b), \cO(\vec{e}_b)) - h^0 \cEnd(\cO(\vec{a}_b))) \\ 
	&= \codim(X_{\vec{a}_b, \vec{e}_b} \subseteq \Quot^{r-s,d-d'}(\pi^*\cF'(-m))).
\end{align*}
\end{proof}

\begin{remark}
Note that both Proposition \ref{prop:irred_onestep} and Proposition \ref{prop:lci} work for arbitrary numerical specifications $s, d', r, d$. In particular they hold when the map $S_{s,d'}(\cE) \to B$ is not birational onto its image, though this is stronger than what we will need.
\end{remark}

\begin{theorem} \label{thm:flag_irred}
Let $B$ be an irreducible scheme of finite type with a good map $\phi_{\cE}: B \to \vbst$. Let $\ur = (r_1,\dots, r_{\ell}), \ud = (d_1, \dots, d_{\ell})$ be such that $r \geq r_1 \geq \cdots \geq r_{\ell} \geq 0$ and $r_i, d_i \in \Z$. Then $\FQuot(B, \cE; \ur, \ud)$ is irreducible. In particular, $\FQuot(\vbst, \cE_{\univ}; \ur, \ud)$ is irreducible.
\end{theorem}
\begin{proof}
We can build $\FQuot(B, \cE; \ur, \ud)$ as a tower of relative Quot schemes. Let $\cS_2$ be the universal subbundle of $\cE$ of rank $r-r_2$, degree $d-d_2$ on $Q \coloneq \FQuot(B, \cE; (r_2,\dots,r_{\ell}), (d_2,\dots,d_{\ell}))$, then $\FQuot(B, \cE; \ur, \ud)$ can also be viewed as $S_{r-r_1,d-d_1}(Q, \cS_2)$. Repeated applications of Proposition \ref{prop:irred_onestep} along this tower shows that $\FQuot(B, \cE; \ur, \ud)$ is irreducible.

Let $Z$ be a component of $\FQuot(\vbst, \cE_{\univ}; \ur, \ud)$. Then the restriction of $\cE$ to $Z$ must generically split as some $\vec{e}$. Thus, if $\FQuot(\vbst, \cE_{\univ}; \ur, \ud)$ has two components where $\cE$ generically splits as $\vec{e}$ and $\pvec{e}'$, then we will detect them when we base change to $U_{\vec{f}}$ where $\vec{f} \leq \vec{e}$ and $\vec{f} \leq \pvec{e}'$. But $U_{\vec{f}} \to \vbst$ is good, so our argument above shows that the base change of $\FQuot(\vbst, \cE_{\univ}; \ur, \ud)$ to $U_{\vec{f}}$ is irreducible for all $\vec{f}$.
\end{proof}

\subsection{Smoothness} Now we will prove smoothness of $\SplittNyeI$ when $I$ is $\vec{e}$-admissible. The key ingredient is an exact sequence that allows us to count the dimension of the tangent space.
\begin{proof}[Proof of smoothness claim in Theorem \ref{thm:smoooth}]
Recall from Section \ref{subsec:sm_cover} and specifically Proposition \ref{prop:moduli_construction} that there is an explicit smooth atlas $\{U_{\pvec{e}'}\} = \{H^1 \cEnd (\cO(\pvec{e}'))\}$ of $\vbst$, indexed by all possible splitting types of rank $r$ and degree $d$, which satisfies the following properties:
\begin{itemize}
  \item $\dim U_{\pvec{e}'} = u(\pvec{e}')$;
  \item $\Sigma_{\pvec{e}'}(U_{\pvec{e}'})$ is a single point, which we call $\{0\}$ (as it is $0$ in an affine space);
  \item the map $U_{\pvec{e}'} \to \vbst$ is good, in particular $\nye{U}_{\vec{e},\pvec{e}'}  = \nye{U}_{I, \vec{e},\pvec{e}'} := \SplittNyeI \times_{\vbst} U_{\pvec{e}'} = \Quot(U_{\pvec{e}'}, \cE_{U_{\pvec{e}'}}; \ur_{I, \HN}, \ud_{I,\HN})$ is irreducible;
  \item $U_{\pvec{e}'}$ maps surjectively onto $\bigcup_{\pvec{f} \geq \pvec{e}'} \SplittOpenF$.
\end{itemize}  
Using this smooth atlas, to prove that $\SplittNyeI$ is smooth, it suffices to prove one of the following two equivalent claims for all $\pvec{e}'$: 
\begin{enumerate}
  \item the restriction of $\SplittNyeI$ to $\bigcup_{\pvec{f} \geq \pvec{e}'} \SplittOpenF$ is smooth; 
  \item the scheme
      \[
      \nye{U}_{\vec{e},\pvec{e}'}  = \nye{U}_{I, \vec{e},\pvec{e}'} := \SplittNyeI \times_{\vbst} U_{\pvec{e}'} = \Quot(U_{\pvec{e}'}, \cE_{U_{\pvec{e}'}}; \ur_{I, \HN}, \ud_{I,\HN})
      \]
is smooth. 
\end{enumerate}

The proof will proceed by induction on the dominance order for $\pvec{e}'$. Since $\SplittNyeI$ surjects onto $\Splitt$, when $\pvec{e}' > \vec{e}$, the restriction of $\SplittNyeI$ to $\bigcup_{\pvec{f} \geq \pvec{e}'} \SplittOpenF$ is empty.

Now assume that both (1) and (2) are true for all $\pvec{e}'' > \pvec{e}'$. Let $\pi: \nye{U}_{\vec{e},\pvec{e}'} \to U_{\pvec{e}'}$. Then since (1) is true for all $\pvec{e}'' > \pvec{e}'$, $U_{\vec{e},\pvec{e}'}$ is smooth away from $\pi^{-1}(\{0\})$. By Theorem \ref{thm:flag_irred}, $\nye{U}_{\vec{e},\pvec{e}'}$ is irreducible, so 
\[
\dim \nye{U}_{\vec{e},\pvec{e}'} = \dim \nye{U}_{\vec{e},\pvec{e}'} \setminus \pi^{-1}(\{0\}) = \dim \SplittNyeI - \dim \vbst+ \dim U_{\pvec{e}'} = u(\pvec{e}') - u(\vec{e}).
\] 
Let $[\alpha]$ be a point in $\pi^{-1}(\{0\})$, corresponding to a flag $\alpha: E_1 \hookrightarrow E_2 \hookrightarrow \cdots \hookrightarrow E_k \hookrightarrow \cO(\pvec{e}')$. To show that $\nye{U}_{\vec{e},\pvec{e}'}$ is everywhere smooth, it suffices to check that for all such $[\alpha] \in \pi^{-1}(\{0\})$, the dimension of the tangent space is
\[
  \dim \T_{[\alpha]} \nye{U}_{\vec{e},\pvec{e}'} = u(\pvec{e}') - u(\vec{e}).
\]

Because the tangent space of $0 \in U_{\pvec{e}'}$ is just $H^1 \cEnd(\cO(\pvec{e}'))$, the space of first-order deformations of $\cO(\pvec{e}')$, the tangent space of $U_{\vec{e},\pvec{e}'}$ at $[\alpha]$ has a modular interpretation:
\[
  \T_{[\alpha]} \nye{U}_{\vec{e},\pvec{e}'} = 
  \left \{
    \begin{tabular}{lll}
      \text{A flag of subsheaves $\cE_{1} \hookrightarrow \cE_2 \hookrightarrow \dots \hookrightarrow \cE_{k+1} = \cE$ on $\bP^1_{k[\varepsilon]/(\varepsilon^2)}$ specializing to $\alpha$} \\
      \text{over $\bP^1_k$, such that each quotient $\cE/\cE_i$ is flat over $k[\varepsilon]/(\varepsilon^2)$}
    \end{tabular}
  \right\} / \sim.
\]
We claim that this tangent space fits inside an exact sequence
\[
0 \to \oplus_{i = 1}^k \End(E_i) 
\to 
\oplus_{i = 1}^{k} \Hom(E_i, E_{i+1}) \to 
\T_{[\alpha]} \nye{U}_{\vec{e},\pvec{e}'}
\to \oplus_{i = 1}^{k+1} H^1 \cEnd (E_i) \to \oplus_{i = 1}^{k} \Ext^1(E_i,E_{i+1}) \to 0.
\]
We will describe each of the maps and argue exactness going from right to left.
\begin{enumerate}
  \item To specify the last map is the same as specifying maps $\phi_{i,j} = H^1 \cEnd (E_i) \to \Ext^1(E_j, E_{j+1})$ for all $1 \leq i \leq k+1, 1 \leq j \leq k$. We set $\phi_{i,j}$ to be zero if $i \neq j$ and $i \neq j+1$. Note that $\Ext^1(E_j, E_{j+1}) = H^1 \cHom(E_j, E_{j+1})$. The map $\phi_{i,i}$ is induced by the map of bundles $\cEnd (E_i) \to \cHom(E_i,E_{i+1})$ by postcomposing with the inclusion $E_i \hookrightarrow E_{i+1}$ in the flag. The map    $\phi_{i,i-1}$ is the negation of the map induced by the map of bundles $\cEnd (E_i) \to \cHom(E_{i-1}, E_i)$, which is defined by precomposing with the inclusion $E_{i-1} \hookrightarrow E_i$.
  Let $i < k$. Then   
  \[
    0 \to \cHom(E_{i+1}/E_i, E_{i+1}) \to \cEnd (E_{i+1}) \to \cHom(E_i, E_{i+1}) \to 0
  \]
  is exact, because $E_{i+1}$ is a vector bundle. Since we are on a curve, the induced map $H^1 \cEnd (E_{i+1}) \to H^1 \cHom(E_i, E_{i+1}) = \Ext^1(E_i, E_{i+1})$ is surjective. This implies that the last map is surjective.
  \item The kernel of the last map corresponds to those deformations of $E_1, \dots, E_k$ that can be lifted to a deformation of the flag. To prove this, it suffices to prove it for a single map, namely if $E \hookrightarrow F$, then the kernel of $H^1 \cEnd (E) \oplus H^1 \cEnd (F) \to \Ext^1(E,F)$ corresponds to those pairs of first-order deformations of $E$ and $F$ that can be lifted to a deformation of $E \hookrightarrow F$. This is the content of Lemma \ref{lem:def_maps} below. The map $\T_{[\alpha]} \nye{U}_{\vec{e},\vec{e}'} \to \oplus_{i = 1}^{k+1} H^1 \cEnd (E_i)$ is the map that forgets the deformations of the inclusions, and only remembers the deformations of each bundle. It clearly maps onto the kernel of the last map, given our interpretation of the kernel.
  \item The kernel of the third map corresponds to those deformations of the flag that do not deform the vector bundles. In other words, there is a map
  \[
    \T_{[\alpha]} \prod_{i = 1}^m \Subsheaves(E_i, E_{i+1}) \to 
    \T_{[\alpha]} \nye{U}_{\vec{e},\pvec{e}'},
  \] 
  and the kernel of the third map is precisely the image of this map. In particular we can think about each $\Subsheaves(E_i, E_{i+1})$ individually. The tangent space of $\Subsheaves(E_i, E_{i+1})$ at $\phi: E_i \hookrightarrow E_{i+1}$ is the cokernel of $H^0 \cEnd(E_i) \to H^0(\cHom(E_i,E_{i+1}))$, induced by post-composing with $\phi$.
  \item The first map is the direct sum of the maps $\End(E_i) \to \Hom(E_i,E_{i+1})$, which are all injective, as they are induced from the injective maps of sheaves $\cEnd(E_i) \to \cHom(E_i,E_{i+1})$.
\end{enumerate}

Using this exact sequence, we can compute $\dim \T_{[\alpha]} \nye{U}_{\vec{e},\pvec{e}'}$:
\begin{align*}
\dim \T_{[\alpha]} \nye{U}_{\vec{e},\pvec{e}'} & = h^1(\cEnd (\cO(\pvec{e}'))) - \sum_{i = 1}^{m} \chi(\cEnd(E_i')) + \sum_{i = 1}^{m} \chi(\cHom(E_i', E_{i+1}')) \\ 
& = u(\pvec{e}') + \sum_{i = 1}^m \chi(\cHom(E_i', E_{i+1}'/E_i')) \\ 
& = u(\pvec{e}') + \sum_{i = 1}^m \chi(\cHom(E_i, E_{i+1}/E_i)) \\ 
& = u(\pvec{e}') - \sum_{i = 1}^m h^1(\cHom(E_i, E_{i+1}/E_i)) \\ 
& = u(\pvec{e}') - u(\vec{e}).
\end{align*}
\end{proof}

\begin{lemma} \label{lem:def_maps}
Let $A$ be an Artinian local ring with residue field $k$, and let $I$ be a square-zero ideal in $A$. Given a map of vector bundles $\phi: E \to F$ on $X \times \Spec A/I$, the deformations of $E, F$ to $\Spec A$ that can be extended to a deformation of $\phi$ is characterized by the kernel of a map $\Ext^1_{X_0}(E_0, E_0 \otimes_k I) \oplus \Ext^1_{X_{0}}(F_0, F_0 \otimes_k I) \to \Ext^1_{X_0}(E_0, F_0 \otimes_k I)$, which is the difference of the natural maps $\Ext^1_{X_0}(E_0, E_0 \otimes_k I) \to \Ext^1_{X_0}(E_0, F_0 \otimes_k I)$ and $\Ext^1_{X_{0}}(F_0, F_0 \otimes_k I) \to \Ext^1_{X_0}(E_0, F_0 \otimes_k I)$ induced by $\phi$.
\end{lemma}
\begin{proof}
Corresponding to deformations $\nye{E}, \nye{F}$ of $E, F$ over $\Spec A$, we have two extensions
\begin{align*}
0 \to E_0 \otimes_k I \to \nye{E} \to E \to 0, \\ 
0 \to F_0 \otimes_k I \to \nye{F} \to F \to 0.
\end{align*}
To deform the map $\phi: E \to F$ into a map of $\nye{\phi}: \nye{E} \to \nye{F}$ is the same as writing down a map of these two short exact sequences such that the maps on the kernel and quotient are induced by $\phi$. This is possible if and only if the pushout of the first sequence by the map $\phi_0 \otimes I: E_0 \otimes_k I \to F_0 \otimes_k I$ agrees with the pullback of the second sequence by the map $\phi: E \to F$. This is if and only if the difference of their classes in $\Ext^1_{X_{A/I}}(E, F_0 \otimes_k I) = \Ext^1_{X_0}(E_0, F_0 \otimes_k I)$ is zero.
\end{proof}

\begin{remark}
One can define Harder-Narasimhan strata (HN strata) inside the stack of vector bundles on a fixed curve $C$ as those vector bundles with a fixed Harder-Narasimhan polygon. These HN strata are generalizations of $\SplittOpen$ to higher genus. One can also define an analogue of $\Splitt$ as the stratum of vector bundles whose HN polygon is a specialization of some fixed polygon. In the literature, the standard way to show that the locally closed HN strata are smooth is via hyper-Quot constructions, similar to $\SplittNyeHN$. Proposition 15.4.2 of \cite{lePotier} gives a general criterion for smoothness for such hyper-Quot constructions, and Corollary 15.4.3 uses this criterion to prove smoothness of HN strata for complete families of vector bundles. The main difference between the case of higher genus curves and $\bP^1$ is that we no longer know if the hyper-Quot construction is irreducible, though it must contain one component which is a resolution of singularities for Harder-Narasimhan strata.
\end{remark}

\subsection{Fibers of the resolution}
To get a better sense of the resolutions $\SplittNyeI$, we give some examples of the fiber of $\SplittNyeHN$ over a point $p \in \mathbf{\Sigma}_{\pvec{e}'}$ when $\pvec{e}' < \vec{e}$. These examples show that the fibers can be reducible with components of different dimensions.
\begin{example}
Let $\vec{e} = (-2,0,2)$, and let $\pvec{e}' = (-3,0,3)$. We can stratify the fiber by how each bundle in a flag splits, and there are two possibilities:
  \[
    \alpha = [\cO(2) \hookrightarrow \cO(0,2) \hookrightarrow \cO(-3,0,3)],
  \]
  \[
    \beta = [\cO(2) \hookrightarrow \cO(-1,3) \hookrightarrow \cO(-3,0,3)].
  \]
Let $X_{\alpha}, X_{\beta}$ denote the collection of flags with such splitting types. $X_{\beta} \cong \bP^1 \times \bP^1$ is closed. $X_{\alpha}$ really is only the data of the second inclusion, which can be identified with quotients $\cO(0,3) \twoheadrightarrow \kappa(p)$ which do not factor through the projection to $\cO$. This is isomorphic to $\bP(\cO(0,3))$ over $\bP^1$ minus the directrix. The closure $\overline{X}_{\alpha}$ can be identified with $\bP(\cO(0,3))$, and is glued to $X_{\beta} = \bP^1 \times \bP^1$ along the directrix of the former and the diagonal of the latter.
\end{example}

\begin{example} Let $\vec{e} = (-1,0,1)$ and $\pvec{e}' = (-2,-1,3)$. In this case, we may have
  \[
    \alpha = [\cO(1) \hookrightarrow \cO(-1,2) \hookrightarrow \cO(-2,-1,3)],
  \]
  \[
    \beta = [\cO(1) \hookrightarrow \cO(-2,3) \hookrightarrow \cO(-2,-1,3)].
  \]
The dimensions are $\dim X_{\alpha} = 3$, $\dim X_{\beta} = 4$. The stratum $X_{\beta}$ is closed and isomorphic to $\bP^2 \times \bP^2$. $X_{\alpha}$ is isomorphic to $\bP^1 \times V$, where $V$ is isomorphic to $\bP(\cO(-1,3))$ minus the directrix. The intersection $\overline{X}_{\alpha} \cap X_{\beta}$ can be identified with those flags of type $\beta$
\[
\cO(1) \xrightarrow{\vectwo{0}{Q}} \cO(-2,3) \xrightarrow{\begin{bmatrix} 0 & 0 \\ L & 0 \\ 0 & 1 \end{bmatrix}} \cO(-2,-1,3)
\]
where $L, Q$ are forms of degrees $1$ and $2$, and $L$ divides $Q$. Since $L$ and $Q$ are defined up to scaling, this intersection is isomorphic to $\bP^1 \times \bP^1$.
\end{example}
%!TEX root = splitting_loci.tex

\section{Rational singularities}
\label{sec:rat_sing}
In this section, we will use the resolution of singularities constructed in the previous section to show that tame splitting loci have rational singularities. Throughout this section, we will assume that we are working over an algebraically closed field of characteristic zero. 

We will need to repeatedly apply Littlewood-Richardson rules and we also make heavy use of the Borel-Weil-Bott theorem. For more details on Littlewood-Richardson rules, we refer the reader to Section 2.3 of \cite{Weyman_2003}, and Theorem 1.4 of \cite{KOUWENHOVEN199177}. For details on Borel-Weil-Bott, the reader may refer to Corollary 4.1.9 of \cite{Weyman_2003}. Given partitions $\mu, \alpha, \beta$, we will use $c_{\alpha, \beta}^{\mu}$ to denote the Littlewood-Richardson coefficient, which, among other things, is the multiplicity of $\bfS_{\alpha} V \otimes \bfS_{\beta} W$ in the decomposition of $\bfS_{\mu}(V \oplus W)$.

Let $\vec{e}$ be a tame splitting type. Then $\Splitt$ admits a resolution of singularities by $\SplittNyeI$, where $|I| = 1$. Let $\SplittNye = \SplittNyeI = \Quot^{r_1, d_1}(\vbst, \cE_{\univ})$. 

Having rational singularities is a smooth local property, so we may assume that we are working with a smooth scheme $B$ of finite type, equipped with a smooth map $B \to \vbst$ corresponding to a vector bundle $\cE$ on $B \times \bP^1$. The reader may refer to Figure \ref{fig:embedding_of_nyesplitt_into_F} for our labelling of the relevant maps. We denote the pullback of $\SplittNye$ to $B$ by $\nyesplitt$. Up to twisting, we may assume that the vector bundle $\cE$ on $B \times \bP^1$ is fiberwise globally generated. By Theorem \ref{thm:LLVnormality}, $\splitt$ is normal, so the natural map $\cO_{\splitt} \to \rho_*\cO_{\nyesplitt}$ is an isomorphism. To show that $\Splitt$ has rational singularities, it remains to show that $R^i \rho_* \cO_{\nyesplitt} = 0$ for $i > 0$. Recall from Proposition \ref{prop:lci} that $\nyesplitt$ admits an embedding into $F = \Quot^{r_1,d_1}(\pi^* \pi_* \cE)$, and it is cut out as a local complete intersection by a section of the vector bundle $\cR = f_* \cHom(\nye{\gamma}^* \pi^* (\pi_* \cE(-1))(-1), \cQ)$. The proposition below shows that in this situation, it suffices to prove the vanishing of certain higher pushforwards of the wedge powers of $\cR^{\vee}$.

\begin{figure}
\begin{tikzcd}
\nyesplitt \arrow[r, "i", hook] \arrow[d, "\rho", two heads] & {F = \Quot^{r_1,d_1}(\pi^*\pi_*\cE)} \arrow[d, "\gamma", two heads] 
& F \times \bP^1 \arrow[l, "f"'] \arrow[d, "\nye{\gamma}"]
\\
\splitt \arrow[r, "j", hook]              & B  & B \times \bP^1 \arrow[l, "\pi"]                                       
\end{tikzcd}
\caption{Embedding of $\nyesplitt$ into $F$}
\label{fig:embedding_of_nyesplitt_into_F}
\end{figure}

\begin{proposition} \label{prop:coh_cond} Suppose that we have the following commutative diagram of schemes, where $i$ and $j$ are closed immersions, and $i$ is a local complete intersection cut out by a section of the vector bundle $\cR$ on $F$. Suppose that $R^{> i} \gamma_* \wedge^i \cR^{\vee} = 0$ for $i \geq 0$. Then $R^{> 0} \rho_* \cO_{\nye{Z}} = 0$.
\[
\begin{tikzcd}
\nye{Z} \arrow[d, "\rho"] \arrow[r, "i", hook] & F \arrow[d, "\gamma"] \\
Z \arrow[r, "j", hook]                                  & B                          
\end{tikzcd}
\]
\end{proposition}

\begin{proof}
The condition tells us that the Koszul complex $\wedge^{\bullet} \cR^{\vee} \to i_* \cO_{\nye{Z}}$ is a resolution of $i_* \cO_{\nye{Z}}$ by locally free $\cO_F$-modules. Let $\cK_i = \ker(\wedge^{i-1} \cR^{\vee} \to \wedge^{i-2} \cR^{\vee})$, so that for $i \geq 1$,
\[
0 \to \cK_{i+1} \to \wedge^{i} \cR^{\vee} \to \cK_{i} \to 0,
\]
and 
\[
0 \to \cK_1 \to \cO_F \to i_* \cO_{\nye{Z}} \to 0.
\]
We claim that $R^{> 1} \gamma_* \cK_{1} = 0$. Let $M = \rk \cR$. By assumption $R^{> M} \gamma_* \cK_{M} = R^{> M} \gamma_* \wedge^M \cR^{\vee} = 0$. Assume that $R^{> i+1} \gamma_* \cK_{i+1} = 0$. Then the first short exact sequence, combined with the fact that $R^{> i}\gamma_* \wedge^i \cR^{\vee} = 0$, shows that $R^{> i} \gamma_* \cK_i = 0$. The claim follows by induction. 

Then using the second exact sequence, we see that $R^{>0} \gamma_*(i_*\cO_{\nye{Z}}) = 0$. The diagram is commutative and pushforwards by closed immersions are exact, so $R^{k} \gamma_*(i_*\cO_{\nye{Z}}) = j_* R^k \rho_* \cO_{\nye{Z}}$ for all $k \geq 0$. Hence $j_* R^{>0} \rho_* \cO_{\nye{Z}} = 0$, which implies that $R^{>0} \rho_* \cO_{\nye{Z}} = 0$. 
\end{proof}

The same idea in the previous proof also proves the following proposition, which we will need later.
\begin{proposition} \label{prop:propagate}
Let $\cV_{\bullet} \to \cF$ be a finite exact sequence of coherent sheaves on a proper scheme $X$. Let $k$ be a nonnegative integer. If for all $i \geq 0$, $H^{\geq i + k}(X, \cV_i) = 0$, then $H^{\geq k}(X, \cF) = 0$. \qed
\end{proposition}

The map $q: F \to B$ is flat with fibers all isomorphic to $\Quott$, where $N = h^0 \cO(\vec{e})$, $\vec{e}$ is any splitting type showing up in $B$. By the theorem below and cohomology and base change, we know that $R^{> 0} \gamma_* \cO_{F} = 0$. So the condition of Proposition \ref{prop:coh_cond} is satisfied when $i = 0$.
\begin{theorem}(\cite{Stromme1987}, Theorem 2.1) $Q = \Quot^{r,d}(\cO_{\bP^1}^{\oplus N})$ is an irreducible, nonsingular, rational projective variety. In particular, $H^{>0}(Q, \cO_Q) = 0$.
\end{theorem}

Therefore, to show the $\splitt$ has rational singularities, by Proposition \ref{prop:coh_cond}, it remains to show that $R^{> n} \gamma_* \bigwedge^n \cR^{\vee} = 0$ for all $n > 0$, where 
\begin{align*}
\cR &= f_* \cHom(\nye{\gamma}^* \pi^* (\pi_* \cE(-1))(-1), \cQ) \\ 
&= (\gamma^* \pi_* \cE(-1))^{\vee} \otimes f_* \cQ(1).
\end{align*} 
In characteristic $0$, Cauchy's formula says that if $\cE, \cF$ are vector bundles, then 
\[
  \bigwedge^n(\cE \otimes \cF) \cong 
  \oplus_{\mu \vdash n} \bfS_{\mu^{\dagger}} \cE \otimes \bfS_{\mu} \cF,
\]
where $\mu^{\dagger}$ is the conjugate partition of $\mu$, and $\mu$ ranges over all partitions of $n$.
Using this identity, we have
\begin{align*}
\bigwedge^n \cR^{\vee} &= \bigwedge^n \left ( (\gamma^* \pi_* \cE(-1)) \otimes (f_* \cQ(1))^{\vee} \right )\\ 
&= \bigoplus_{\mu \vdash n} \bfS_{\mu^{\dagger}} (\gamma^* \pi_* \cE(-1)) \otimes \bfS_{\mu} ((f_* \cQ(1))^{\vee}) \\ 
&= \bigoplus_{\mu \vdash n}  \gamma^* \bfS_{\mu^{\dagger}} (\pi_* \cE(-1)) \otimes \bfS_{\mu} ((f_* \cQ(1))^{\vee}).
\end{align*}
So by the projection formula,
\begin{align*}
R^k \gamma_* \bigwedge^n \cR^{\vee} &=
\bigoplus_{\mu \vdash n}  \bfS_{\mu^{\dagger}} (\pi_* \cE(-1)) \otimes R^k \gamma_*(\bfS_{\mu} ((f_* \cQ(1))^{\vee})).
\end{align*}

Therefore, to prove Theorem \ref{thm:rational}, it suffices to show that when $k > n = |\mu|$, the higher pushforwards $R^k \gamma_*(\bfS_{\mu} ((f_* \cQ(1))^{\vee}))$ are zero. Since $\gamma$ is flat with fibers all isomorphic to $\Quott$, by cohomology and base change, it suffices to show vanishing for the cohomology of corresponding bundles on $\Quott$.

%!TEX root = splitting_loci.tex

Let $d > 0$ and $N \geq r \geq 0$. Let $\cQ$ be the universal quotient on $\Quott \times \bP^1$, and let $p: \Quott \times \bP^1 \to \Quott$ be the projection to the first factor. Recall from the introduction that on $\Quott$, for all $m \geq -1$, we define tautological bundles $\cE_m = (p_* \cQ(m))^{\vee}$.

\begin{theorem} \label{thm:used_vanishing}
Let $\lambda$ be any nonempty partition. Then $H^{> |\lambda|}(\Quott, \bfS_{\lambda} \cE_1) = 0$.
\end{theorem}
By our previous discussion, Theorem \ref{thm:used_vanishing} implies Theorem \ref{thm:rational}. It follows from the more general statement below.
\begin{theorem}[Restatement of Theorem \ref{thm:general_vanishing}]
Let $m \geq -1$ be an integer. Let $\lambda_{-1}, \lambda_0, \lambda_1, \dots, \lambda_m$ be partitions, at least one of which is nonempty. Let $D = |\lambda_{-1}| + \cdots + |\lambda_m|$. Then 
\[
	H^{> D}(\Quott, \otimes_{i = -1}^m \bfS_{\lambda_i} \cE_i) = 0.
\]
\end{theorem}

We will prove Theorem \ref{thm:general_vanishing} above by induction on the number of factors in the tensor product. For the base case, we will need the following proposition.
\begin{proposition} \label{prop:twoLambdas} Let $D > 0$ be a positive integer. Then for sufficiently large $m$, for any two partitions $\alpha, \beta$ such that $|\alpha| + |\beta| = D$, 
\[
	H^{> D}(\Quott, \bfS_{\alpha} \cE_{m-1} \otimes \bfS_{\beta} \cE_{m}) = 0.
\]
\end{proposition}

\begin{proof}[Proof of Theorem \ref{thm:general_vanishing} assuming Proposition \ref{prop:twoLambdas}]
On $\Quott \times \bP^1$, for any $n \in \Z$, we have an exact sequence
\[
	0 \to \cO(n) \to \cO(n+1)^{\oplus 2} \to \cO(n+2) \to 0,
\] 
by pulling back the same sequence from $\bP^1$. Tensoring this sequence with $\cQ$, since $\cO(n+2)$ is locally free, the sequence remains exact:
\[
	0 \to \cQ(n) \to \cQ(n+1)^{\oplus 2} \to \cQ(n+2) \to 0.
\] 
For $n \geq -1$, $R^1 p_* \cQ(n) = 0$, so the sequence remains exact after taking $p_*$, and each term is a vector bundle on $\Quott$:
\[
	0 \to p_*\cQ(n) \to p_*\cQ(n+1)^{\oplus 2} \to p_*\cQ(n+2) \to 0.
\]
Dualizing, we have
\begin{equation} \label{eq:propogate}
0 \to \cE_{n+2} \xrightarrow{\Phi_{n+1}} \cE_{n+1}^{\oplus 2} \to \cE_n \to 0.
\end{equation}
We will prove Theorem \ref{thm:general_vanishing} for a fixed $D > 0$ by induction on the index $i$ of the smallest nonempty partition $\lambda_i$. Without loss of generality, we may take $m$ to be sufficiently large. So Proposition \ref{prop:twoLambdas} takes care of the base cases $i = m$ and $i = m-1$. Now suppose that Theorem \ref{thm:general_vanishing} is true when the smallest nonempty partition is $\lambda_k$. There is a complex $(\bfS_{\lambda_{k-1}} \Phi_k)_{\bullet}$ formed on the map $\Phi_k: \cE_{k+1} \to \cE_k^{\oplus 2}$, called the Schur complex, whose $t$-th term is \[
	(\bfS_{\lambda_{k-1}} \Phi_k)_{t} = \oplus_{|\alpha| = t, c_{\alpha, \beta}^{\lambda_{k-1}} \neq 0} \bfS_{\alpha^{\dagger}} \cE_{k+1} \otimes \bfS_{\beta} \cE_k^{\oplus 2}.
\]
Corollary V.1.15 of \cite{KaanBuchsbaumWyman} tells us that when the cokernel of $\Phi_k$ is locally free, this complex is a resolution of the Schur functor of the cokernel, which is the vector bundle $\bfS_{\lambda_{k-1}} \cE_{k-1}$ in our case. 

Tensoring this resolution with $\otimes_{i = k}^m \bfS_{\lambda_i} \cE_i$ and applying Littlewood-Richardson rules, we find that the vector bundle $\otimes_{i = k-1}^m \bfS_{\lambda_i} \cE_i$ has a resolution whose $t$-th term is the direct sum of bundles of the form $\otimes_{i = k+2}^m \bfS_{\lambda_i} \cE_i \otimes \bfS_{\lambda_{k+1}'} \cE_{k+1} \otimes \bfS_{\lambda_k'} \cE_k$, where $|\lambda_{k+1}'| = |\lambda_{k+1}| + t$, and $|\lambda_{k}'| = |\lambda_k| + |\lambda_{k-1}| - t$. In particular, the sum of the sizes of the partitions involved is $D$ for all bundles that show up in the resolution. The inductive hypothesis applies to show that all terms in the resolution have no cohomology in degree greater than $D$, which implies that $H^{>D}(\otimes_{i = k-1}^m \bfS_{\lambda_i} \cE_i) = 0$.
\end{proof}

\subsection{Proof of Proposition \ref{prop:twoLambdas}}
For every integer $m \geq d-1$, we have a short exact sequence 
\[
  0 \to p_* \cS(m) \to p_* \cO(m)^{\oplus N} \to p_* \cQ(m) \to 0
\]
of vector bundles on $\Quott$, inducing a map $g_m: \Quott \to G_m$, where $G_m = \Gr(\rk p_* \cS(m), N(m+1))$. Str{\o}mme proves that for $m \geq d$, the map $\iota = \iota_{m-1} = (g_{m-1}, g_m): \Quott \to G_{m-1} \times G_m$ is a closed embedding (Theorem 4.1, \cite{Stromme1987}). Let 
\begin{align}
  & 0 \to \cA_1 \to \cO_{G_{m-1}}^{\oplus Nm} \to \cB_1 \to 0 \\ 
  & 0 \to \cA_2 \to \cO_{G_{m}}^{\oplus N(m+1)} \to \cB_2 \to 0
\end{align}
be the universal sequences on $G_{m-1}$ and $G_m$, and let $W = H^0(\bP^1, \cO(1))$. For $i = 1, 2$, let $k_i = \rk \cA_i, n_1 = Nm, n_2 = N(m+1)$, so that $G_{m-1} \times G_m = \Gr(k_1, n_1) \times \Gr(k_2,n_2)$. Str{\o}mme also proves that $\Quott$ is the zero locus of a global section of the vector bundle $\cA_1^{\vee} \boxtimes (\cB_2 \otimes W) \cong \cA_1^{\vee} \boxtimes (\cB_2 \oplus \cB_2)$, whose rank is equal to the codimension of $\Quott$ in $G_{m-1} \times G_m$, exhibiting it as a local complete intersection (Theorem 4.1, \cite{Stromme1987}).

As a result, $\iota_* \cO_{\Quott}$ admits a Koszul resolution as a $\cO_{G_{m-1} \times G_m}$-module: 
\[
\wedge^{\bullet} \left ( 
\cA_1 \boxtimes (\cB_2 \oplus \cB_2)^{\vee} 
\right )
\to \iota_* \cO_{\Quott}.
\]
By definition of $g_{m-1}$ and $g_m$, we know that 
$\iota^* (\bfS_{\alpha} \cB_1^{\vee} \boxtimes \bfS_{\beta}\cB_2^{\vee}) 
= \bfS_{\alpha} \cE_{m-1} \otimes \bfS_{\beta} \cE_{m}$. 
So tensoring the Koszul resolution with $\bfS_{\alpha} \cB_1^{\vee} \boxtimes \bfS_{\beta} \cB_2^{\vee}$, we get a resolution of $\iota_*(\bfS_{\alpha} \cE_{m-1} \otimes \bfS_{\beta} \cE_{m})$:
\[
  \wedge^{\bullet} \left ( 
\cA_1 \boxtimes (\cB_2 \oplus \cB_2)^{\vee} 
\right ) \otimes 
(\bfS_{\alpha} \cB_1^{\vee} \boxtimes \bfS_{\beta} \cB_2^{\vee})
\to \iota_* (\bfS_{\alpha} \cE_{m-1} \otimes \bfS_{\beta} \cE_{m}).
\]
For a given $D = |\alpha| + |\beta|$, we will choose $m$ to be sufficiently large so that $D < n_1 - k_1$. In particular, any partition $\alpha$ with $|\alpha| \leq D$ must have $\alpha_{n_1-k_1} = 0$. Let $\mu = (\mu_1,\dots, \mu_{k_1}), \alpha = (\alpha_1, \dots, \alpha_{n_1-k_1})$ and $\beta$ be partitions. 
Let 
\[
	\cV_{\mu, \alpha, \beta} = (\bfS_{\mu}\cA_1 \otimes \bfS_{\alpha}\cB_1^{\vee}) 
	\boxtimes 
	(\bfS_{\mu^{\dagger}}(\cB_2^{\vee} \oplus \cB_2^{\vee}) \otimes \bfS_{\beta}\cB_2^{\vee}).
\]
By Cauchy's formula, 
\[
	\wedge^k \left ( 
\cA_1 \boxtimes (\cB_2 \oplus \cB_2)^{\vee} 
\right ) \otimes 
(\bfS_{\alpha} \cB_1^{\vee} \boxtimes \bfS_{\beta} \cB_2^{\vee}) = 
\oplus_{\mu \vdash k} \cV_{\mu, \alpha, \beta}.
\]

By Proposition \ref{prop:propagate}, the following result implies Proposition \ref{prop:twoLambdas}.
\begin{proposition} \label{prop:vanishingGrassmannian} 
For all partitions $\alpha, \beta$ with $|\alpha|+|\beta| > 0$ and $\alpha_{n_1-k_1} = 0$, and any partition $\mu$,
\[H^{> |\mu|+|\alpha|+|\beta|}(\cV_{\mu,\alpha,\beta}) = 0.\]
\end{proposition}

% Now we proceed with the proof of Proposition \ref{prop:vanishingGrassmannian}. 
We will first deal with the case where $\mu \neq \emptyset$, which is harder. Suppose that there exist partitions $\gamma_1, \gamma_2, \nu', \nu = (\nu_1, \dots, \nu_{n_2-k_2})$ such that $c_{\gamma_1, \gamma_2}^{\mu^{\dagger}} c_{\gamma_1,\gamma_2}^{\nu'} c_{\nu', \beta}^{\nu} \neq 0$. In other words, 
\[\cW_{\mu, \alpha, \nu} = 
(\bfS_{\mu} \cA_1\otimes \bfS_{\alpha}\cB_1^{\vee})
\boxtimes 
\bfS_{\nu}\cB_2^{\vee}
\]
arises as a summand of $\cV_{\mu, \alpha, \beta}$. By Borel-Weil-Bott  and K\"unneth, if $\cW_{\mu, \alpha, \nu}$ has cohomology, the vector bundles $\bfS_{\mu}\cA_1 \otimes \bfS_{\alpha}\cB_1^{\vee}$ and $\bfS_{\nu}\cB_2^{\vee}$ must both have cohomology in some unique degrees $D_1, D_2$. Then $\cW_{\mu, \alpha, \nu}$ has cohomology in degree $D = D_1 + D_2$. When this happens, Borel-Weil-Bott imposes some combinatorial conditions on the partitions involved. We would like to show that $D \leq |\mu| + |\alpha| + |\beta|$.

Before proceeding further with the proof, note the following example which shows that this bound is sharp in general.
\begin{example}
The Quot scheme $\Quot^{1,3}(\cO^{\oplus 3})$ admits an embedding $\iota_2: \Quot^{1,3}(\cO^{\oplus 3}) \hookrightarrow \Gr(3,9) \times \Gr(5,12)$. We can take $\mu = (7, 2), \alpha = (1), \beta = (2)$. Then $H^{12}(\cV_{\mu,\alpha,\beta}) \neq 0$. Indeed, on the first factor, cohomology degree equals $7$. On the second factor, nonzero cohomology arises from the partition $\nu = (6, 1^5)$, which has cohomology in degree $5$. It comes from combining $\nu' = (4, 1^5)$ and $\beta = (2)$ using the Littlewood-Richardson rule. The possible pairs $\gamma_1, \gamma_2$ such that $c_{\gamma_1, \gamma_2}^{\mu^{\dagger}} c_{\gamma_1, \gamma_2}^{\nu'} \neq 0$ include 
  $(\gamma_1 = (2), \gamma_2 = (2, 1^5)),
    (\gamma_1 = (2, 1), \gamma_2 = (2, 1^4)),
    (\gamma_1 = (2, 1, 1), \gamma_2 = (2, 1^3))$, as well as the three tuples obtained from the above by switching $\gamma_1$ and $\gamma_2$. For all six pairs, $c_{\gamma_1, \gamma_2}^{\mu^{\dagger}} c_{\gamma_1,\gamma_2}^{\nu'} c_{\nu', \beta}^{\nu} = 1$.

Due to a lack of understanding of the map on cohomology induced by the Koszul resolution, we are unable to check whether this nonvanishing group implies that $H^{3}(\Quot, \cE_2 \otimes \text{Sym}^2 \cE_3) \neq 0$.
\end{example}

Now let us delve into the combinatorics further. The vector bundle $\bfS_{\nu}\cB_2^{\vee}$ has cohomology if and only if the following tuple obtained by pointwise addition
\[(0^{k_2}, \nu) + (n_2-1, n_2-2, \dots, 1, 0)\]
has no repeated entries. Borel-Weil-Bott says that if this tuple has no repetitions, then $\bfS_{\nu}\cB_2^{\vee}$ has cohomology in degree equal to the length of the permutation that sorts this vector in descending order, and no cohomology in other degrees. Since $|\mu| > 0$, $|\nu| = |\mu| + |\beta| > 0$. So the condition that the vector has no repetition is equivalent to the existence of some $1 \leq j \leq n_2-k_2$ such that $\nu_j \geq k_2 + j$ and $\nu_{j+1} \leq j$. When such $j$ exists, it is unique and $\bfS_{\nu}\cB_2^{\vee}$ has cohomology in degree $D_2 = jk_2$. 

The cohomology of the first vector bundle $\bfS_{\mu}\cA_1 \otimes \bfS_{\alpha}\cB_1^{\vee}$ is a bit more complicated to describe. Define 
\begin{align*}
f_i &= n_1-k_1+i-1-\mu_i, \\ 
g_j &= n_1-k_1-j+\alpha_j.
\end{align*} 
Let $- \mu = (-\mu_{k_1}, \dots, -\mu_2,-\mu_1)$.
Then Borel-Weil-Bott says that $\bfS_{\mu} \cA_1 \otimes \bfS_{\alpha}\cB_1^{\vee}$ has cohomology if and only if the vector 
\[
(-\mu, \alpha) + (n_1-1, n_1-2, \dots, 1, 0) = 
(f_{k_1},f_{k_1-1},\dots,f_1,g_1,g_2, \dots,g_{n_1-k_1})\] 
does not have repeated entries. In this case, $\bfS_{\mu} \cA_1 \otimes \bfS_{\alpha}\cB_1^{\vee}$ has cohomology  in degree equal to the length of the permutation that sorts this vector in strictly descending order. 

Define indices $0 \leq i_1 \leq i_2 \leq \cdots \leq i_{n_1-k_1} \leq k_1$, where $i_\ell$ is the largest index such that $f_{i_\ell} < g_{n_1-k_1-\ell+1}$ for each $1 \leq \ell \leq n_1-k_1$. When there doesn't exist such an index, we set $i_{\ell} = 0$.
% \begin{align*}
% 	f_{k_1} > f_{k_1-1} > \cdots > f_{i_{n_1-k_1}+1} > g_1 > \\
% 	 f_{i_{n_1-k_1}} > f_{i_{n_1-k_1}-1} > \cdots > f_{i_{n_1-k_1-1} + 1} > g_2 > \\
% 	 \hdots \\
% 	 f_{i_2} > f_{i_2-1} > \cdots > f_{i_1+1} > g_{n_1-k_1} > \\
% 	 f_{i_1} > f_{i_1-1} > \cdots > f_{1}.
% \end{align*}
In terms of $i_1, \dots, i_{n_1-k_1}$, the unique degree in which $\bfS_{\mu} \cA_1 \otimes \bfS_{\alpha}\cB_1^{\vee}$ has nonzero cohomology is 
\[
	D_1 = \sum_{\ell = 1}^{n_1-k_1}(i_\ell - i_{\ell-1})(n_1-k_1-\ell+1),
\]
where we let $i_0 = 0$.
% For the vector to have no repeated entries and for the prescribed inequalities to hold, the partitions $\mu, \alpha$ and indices $i_j$ must satisfy some conditions. 
For $1 \leq \ell \leq n_1-k_1$, the inequality $f_{i_{\ell}} < g_{n_1-k_1-\ell+1}$ is equivalent to  
\begin{align}
\mu_{i_{\ell}} - (n_1-k_1-\ell+1) \geq i_{\ell} - \alpha_{n_1-k_1-\ell+1}. \label{eq:ineq2}
\end{align}
Setting $\ell = 1$, we note that \eqref{eq:ineq2} says that 
\begin{equation} \label{eq:mui1}
	\mu_{i_1} \geq n_1-k_1 +i_1-\alpha_{n_1-k_1} = n_1-k_1 +i_1.
\end{equation}
When $i_{\ell} < i_{\ell+1}$, we also have $f_{i_{\ell}+1} > g_{n_1-k_1-\ell+1}$, which is equivalent to 
\begin{align}
i_{\ell}- \alpha_{n_1-k_1-\ell+1} > \mu_{i_{\ell}+1} - (n_1-k_1-\ell+1) \label{eq:ineq1}.
\end{align}
Replacing $\ell$ with $\ell+1$ in \eqref{eq:ineq2}, we find 
\[
	\mu_{i_{\ell+1}} - (n_1-k_1-\ell) \geq i_{\ell+1}-\alpha_{n_1-k_1-\ell}.
\]
Combined with \eqref{eq:ineq1}, we obtain
\[
	i_{\ell} - \alpha_{n_1-k_1-\ell+1} \geq \mu_{i_{\ell}+1}-(n_1-k_1-\ell) \geq \mu_{i_{\ell+1}} - (n_1-k_1-\ell) \geq i_{\ell+1}-\alpha_{n_1-k_1-\ell}.
\]
So for $1 \leq \ell < n_1-k_1$, whenever $i_{\ell} < i_{\ell+1}$,
\begin{equation} \label{eq:diff-of-i-indices}
	 \alpha_{n_1-k_1-\ell} - \alpha_{n_1-k_1-\ell+1} \geq i_{\ell+1}-i_{\ell}.
\end{equation}
If $i_{\ell} = i_{\ell+1}$, then this inequality follows from $\alpha$ being a partition. So we've shown that \eqref{eq:diff-of-i-indices} holds unconditionally for $1 \leq \ell < n_1-k_1$.

To summarize,
\begin{proposition} \label{prop:upperboundforcohdeg}
If $H^{\bullet}(\cW_{\mu,\alpha,\nu}) \neq 0$, then there exist $1 \leq j \leq n_2-k_2$, and $0 = i_0 \leq i_1 \leq \cdots \leq i_{n_1-k_1} \leq k_1$ such that the conditions \eqref{eq:mui1} and \eqref{eq:diff-of-i-indices} are satisfied, and the unique degree in which the cohomology of $\cW_{\mu,\alpha,\nu}$ is nonzero is $D = jk_2 + \sum_{\ell = 1}^{n_1-k_1}(i_\ell - i_{\ell-1})(n_1-k_1-\ell+1)$. Moreover,
\[
	D \leq i_1(n_1-k_1)+jk_2 + |\alpha|.
\]
\end{proposition}
\begin{proof}
We've already proven the proposition except for the last inequality. Recall that $m$ was chosen so that $\alpha_{n_1-k_1} = 0$. Now using Equation \eqref{eq:diff-of-i-indices},
\begin{align*}
	D &= jk_2 + \sum_{\ell = 1}^{n_1-k_1}(i_\ell - i_{\ell-1})(n_1-k_1-\ell+1) \\ 
	&\leq jk_2 + i_1(n_1-k_1) + \sum_{\ell = 2}^{n_1-k_1}(\alpha_{n_1-k_1-\ell+1} - \alpha_{n_1-k_1-\ell+2})(n_1-k_1-\ell+1) \\ 
	&= jk_2 + i_1(n_1-k_1) + |\alpha| - \alpha_{n_1-k_1}(n_1-k_1) \\ 
	&= jk_2 + i_1(n_1-k_1) + |\alpha|.
\end{align*}
\end{proof}

\begin{proposition} \label{prop:lowerboundforsumofparts}Suppose that after applying the Littlewood-Richardson decompositions, $\cV_{\mu, \alpha, \beta}$ contains $\cW_{\mu, \alpha, \nu}$ as a summand, i.e. there exist partitions $\gamma_1, \gamma_2, \nu', \nu$ such that $c_{\gamma_1, \gamma_2}^{\mu^{\dagger}} c_{\gamma_1,\gamma_2}^{\nu'} c_{\nu', \beta}^{\nu} \neq 0$. Suppose $\cW_{\mu, \alpha, \nu}$ has cohomology with associated indices $j, i_\ell$ given by the previous proposition. Then
\[
	|\mu| + |\alpha| + |\beta| \geq jk_2 + i_1(n_1-k_1) + (i_1-j)^2 + |\alpha|.
\]
\end{proposition}
\begin{proof}
Write $\alpha + \beta$ for the partition with $(\alpha + \beta)_i = \alpha_i + \beta_i$, and write $\alpha \cup \beta$ for the partition obtained by stacking the Young diagrams of $\alpha$ and $\beta$ and sorting the rows according to length. As explained in Proposition 3.2.1 of \cite{Marian_Oprea_Sam_2023}, Littlewood-Richardson rules imply that $\nu'$ is dominated by $\alpha + \beta$, and $\mu^{\dagger}$ is dominated by $\alpha \cup \beta$. Therefore
\[
	\mu_1^{\dagger} + \cdots + \mu_{2j}^{\dagger} \geq \alpha_1 + \cdots \alpha_j + \beta_1 + \cdots + \beta_j \geq \nu'_1 + \cdots + \nu'_j.
\]
But since $\nu_j \geq k_2 + j$, we also have that
\[
	\nu_1' + \cdots + \nu_j' \geq \nu_1 + \cdots + \nu_j - |\beta| \geq j(k_2 + j) - |\beta|.
\]
Recall also from \eqref{eq:mui1} that $\mu_{i_1} \geq n_1-k_1+i_1$. So 
\begin{align*}
|\mu| &\geq \sum_{k = 1}^{2j} \mu_k^{\dagger} +  i_1(n_1-k_1+i_1-2j) \\ 
& \geq j(k_2 + j) - |\beta| + i_1(n_1-k_1+i_1-2j) \\ 
& \geq jk_2 + i_1(n_1+k_1) + (i_1-j)^2 - |\beta|.
\end{align*}
Adding $|\alpha| + |\beta|$ to both sides yields the desired claim.
\end{proof}
Now Proposition \ref{prop:lowerboundforsumofparts} and Proposition \ref{prop:upperboundforcohdeg} together show that $|\mu| + |\alpha| + |\beta| \geq D$, which proves Proposition \ref{prop:vanishingGrassmannian} when $\mu \neq \emptyset$. 
It remains to prove Proposition \ref{prop:vanishingGrassmannian} when $\mu = \emptyset$. In this case,
$
	\cV_{\emptyset, \alpha, \beta} = \bfS_{\alpha}\cB_1^{\vee} \boxtimes \bfS_{\beta}\cB_2^{\vee}.
$
The same analysis of $\bfS_{\nu}\cB_{2}^{\vee}$ in the previous case applies. If $\alpha$, $\beta$ are both non-empty, the bundle $\cV_{\emptyset, \alpha, \beta}$ has cohomology if and only if there exist $1 \leq i \leq n_1-k_1$, $1 \leq j \leq n_2-k_2$, such that $\alpha_i \geq k_1 + i$, $\alpha_{i+1} \leq i$, $\beta_{j} \geq k_2 + j$, $\beta_j \leq j$. If this is the case, then $\cV_{\emptyset, \alpha, \beta}$ has cohomology in degree $D = ik_1 + jk_2$. Now
$
	|\alpha| + |\beta| \geq i(k_1+i) + j(k_2+j) > D.
$
The argument in the case where only one of $\alpha$, $\beta$ is nonempty is completely analogous. This concludes the proof of Proposition \ref{prop:vanishingGrassmannian} and the proof of Theorem \ref{thm:rational}.

\printbibliography

@article {KaanBuchsbaumWyman,
    AUTHOR = {Akin, Kaan and Buchsbaum, David A. and Weyman, Jerzy},
     TITLE = {Schur functors and {S}chur complexes},
   JOURNAL = {Adv. in Math.},
  FJOURNAL = {Advances in Mathematics},
    VOLUME = {44},
      YEAR = {1982},
    NUMBER = {3},
     PAGES = {207--278},
      ISSN = {0001-8708},
   MRCLASS = {20C30 (15A72 15A75 15A78)},
  MRNUMBER = {658729},
MRREVIEWER = {A.\ Kh.\ Kushkule\u i},
       DOI = {10.1016/0001-8708(82)90039-1},
       URL = {https://doi.org/10.1016/0001-8708(82)90039-1},
}

@book{Eisenbud1995,
	author = {David Eisenbud},
	doi = {10.1007/978-1-4612-5350-1},
	edition = {1},
	isbn = {978-0-387-94268-1},
	note = {Published: 30 March 1995},
	publisher = {Springer New York, NY},
	series = {Graduate Texts in Mathematics},
	title = {Commutative Algebra: with a View Toward Algebraic Geometry},
	volume = {150},
	year = {1995}}

@misc{Coppens2025,
	archiveprefix = {arXiv},
	author = {Marc Coppens},
	eprint = {2504.21141},
	primaryclass = {math.AG},
	title = {A picture of the irreducible components of $W^r_d(C)$ for a general $k$-gonal curve $C$},
	url = {https://arxiv.org/abs/2504.21141},
	year = {2025}}

@article{larson_degeneracy,
	author = {Hannah K. Larson},
	journal = {Advances in Mathematics},
	month = {March},
	number = {107563},
	title = {Universal degeneracy classes for vector bundles on $\mathbb{P}^1$ bundles},
	volume = {380},
	year = {2021}}

@article{Conrad_2007,
	author = {Conrad, Brian},
	doi = {10.1017/S1474748006000089},
	journal = {Journal of the Institute of Mathematics of Jussieu},
	number = {2},
	pages = {209--278},
	title = {Arithmetic moduli of generalized elliptic curves},
	volume = {6},
	year = {2007}}

@book{Weyman_2003,
	author = {Weyman, Jerzy},
	collection = {Cambridge Tracts in Mathematics},
	place = {Cambridge},
	publisher = {Cambridge University Press},
	series = {Cambridge Tracts in Mathematics},
	title = {Cohomology of Vector Bundles and Syzygies},
	year = {2003}}

@book{lePotier,
	author = {Joseph Le Potier},
	publisher = {Cambridge University Press},
	series = {Cambridge Studies in Advanced Mathematics},
	title = {Lectures on vector bundles},
	volume = {54},
	year = {1997}}

@misc{stacks-project,
	author = {The {Stacks Project Authors}},
	howpublished = {\url{https://stacks.math.columbia.edu}},
	shorthand = {Stacks},
	title = {\textit{Stacks Project}},
	year = {2018}}

@article{CONCA2018111,
	author = {Aldo Conca and Maral Mostafazadehfard and Anurag K. Singh and Matteo Varbaro},
	doi = {10.1016/j.aim.2018.06.011},
	issn = {0001-8708},
	journal = {Advances in Mathematics},
	pages = {111-129},
	title = {Hankel determinantal rings have rational singularities},
	volume = {335},
	year = {2018}}

@article{KOUWENHOVEN199177,
	author = {Frank M Kouwenhoven},
	doi = {https://doi.org/10.1016/0001-8708(91)90020-8},
	issn = {0001-8708},
	journal = {Advances in Mathematics},
	number = {1},
	pages = {77-113},
	title = {Schur and Weyl functors},
	url = {https://www.sciencedirect.com/science/article/pii/0001870891900208},
	volume = {90},
	year = {1991}}

@article{Marian_Negut_2024,
	abstractnote = {We define a categorical action of the shifted quantum loop group of sl2 on the derived categories of Quot schemes of finite length quotient sheaves on a smooth projective curve. As an application, we obtain a semi-orthogonal decomposition of the derived categories of Quot schemes, of representation theoretic origin. We use this decomposition to calculate the cohomology of interesting tautological vector bundles over the Quot scheme.},
	author = {Marian, Alina and Negu{\c t}, Andrei},
	month = nov,
	number = {arXiv:2411.08695},
	title = {Derived categories of Quot schemes on smooth curves and tautological bundles},
	year = {2024}}

@article{Oprea_Sinha_2022,
	abstractnote = {We compute the Euler characteristics of tautological vector bundles and their exterior powers over the Quot schemes of curves. We give closed-form expressions over punctual Quot schemes in all genera. For higher rank quotients of a trivial vector bundle, we obtain answers in genus zero. We also study the Euler characteristics of the symmetric powers of the tautological bundles, for rank zero quotients.},
	author = {Oprea, Dragos and Sinha, Shubham},
	doi = {10.48550/arXiv.2207.01675},
	month = jul,
	title = {Euler characteristics of tautological bundles over Quot schemes of curves},
	year = {2022}}

@article{Marian_Oprea_Sam_2023,
	abstractnote = {We consider tautological bundles and their exterior and symmetric powers on the Quot scheme over the projective line. We prove and conjecture several statements regarding the vanishing of their higher cohomology, and we describe their spaces of global sections via tautological constructions. To this end, we make use of the embedding of the Quot scheme as an explicit local complete intersection in the product of two Grassmannians, studied by Str{\o}mme. This allows us to construct resolutions with vanishing cohomology for the tautological bundles and their exterior and symmetric powers.},
	author = {Marian, Alina and Oprea, Dragos and Sam, Steven V.},
	month = jun,
	number = {arXiv:2211.03901},
	title = {On the cohomology of tautological bundles over Quot schemes of curves},
	year = {2023}}

@article{Sinha_Zhang_2024,
	abstractnote = {We derive a $K$-theoretic analogue of the Vafa--Intriligator formula, computing the (virtual) Euler characteristics of vector bundles over the Quot scheme that compactifies the space of degree $d$ morphisms from a fixed projective curve to the Grassmannian $mathrm{Gr}(r,N)$. As an application, we deduce interesting vanishing results, used in Part I (arXiv:2406.12191) to study the quantum $K$-ring of $mathrm{Gr}(r,N)$. In the genus-zero case, we prove a simplified formula involving Schur functions, consistent with the Borel-Weil-Bott theorem in the degree-zero setting. These new formulas offer a novel approach for computing the structure constants of quantum $K$-products.},
	author = {Sinha, Shubham and Zhang, Ming},
	doi = {10.48550/arXiv.2410.23486},
	month = dec,
	number = {arXiv:2410.23486},
	title = {Quantum $K$-invariants via Quot schemes II},
	year = {2024}}

@inproceedings{Stromme1987,
	address = {Berlin, Heidelberg},
	author = {Str{\o}mme, Stein Arild},
	booktitle = {Space Curves},
	editor = {Ghione, Franco and Peskine, Christian and Sernesi, Edoardo},
	isbn = {978-3-540-47708-2},
	pages = {251--272},
	publisher = {Springer Berlin Heidelberg},
	title = {On parametrized rational curves in grassmann varieties},
	year = {1987}}

@article{Eisenbud-Schreyer,
	author = {David Eisenbud and Frank-Olaf Schreyer},
	journal = {Transactions of the American Mathematical Society},
	month = {October},
	number = {10},
	pages = {5367-5396},
	title = {Relative Beilinson monad and direct image for families of coherent sheaves},
	volume = {360},
	year = {2008}}

@article{LLV,
	author = {Eric Larson and Hannah K. Larson and Isabel Vogt},
	journal = {Geometry \& Topology},
	month = {January},
	number = {1},
	pages = {193-257},
	title = {Global Brill-Noether theory over the Hurwitz space},
	volume = {29},
	year = {2025}}

@article{Larson2021,
	author = {Larson, Hannah K.},
	doi = {10.1007/s00222-020-01023-z},
	id = {Larson2021},
	isbn = {1432-1297},
	journal = {Inventiones mathematicae},
	month = {June},
	number = {3},
	pages = {767--790},
	title = {A refined Brill--Noether theory over Hurwitz spaces},
	volume = {224},
	year = {2021}}

\end{document}